\documentclass[psamsfonts,reqno]{amsart}
\usepackage{amssymb,eucal,xy}
\xyoption{all}

\newcommand{\cml}{complemented modular lattice}
\newcommand{\scml}{sectionally complemented modular lattice}
\newcommand{\Ban}{Ba\-na\-schew\-ski}

\numberwithin{equation}{section}

\newcommand{\pup}[1]{\textup{(}#1\textup{)}}

\theoremstyle{plain}

\newtheorem{lemma}{Lemma}[section]
\newtheorem{theorem}[lemma]{Theorem}
\newtheorem{proposition}[lemma]{Proposition}
\newtheorem{corollary}[lemma]{Corollary}

\newtheorem{claim}{Claim}

\newtheorem*{sclaim}{Claim}

\newtheorem*{stat}{\name}
\newcommand{\name}{testing}

\theoremstyle{definition}
\newtheorem{definition}[lemma]{Definition}

\theoremstyle{remark}
\newtheorem{remark}[lemma]{Remark}

\newenvironment{all}[1]{\renewcommand{\name}{#1}\begin{stat}}
                        {\end{stat}}

\newcommand{\qedc}{{\qed}~{\rm Claim~{\theclaim}.}}
\newcommand{\qedsc}{{\qed}~{\rm Claim.}}

\newenvironment{cproof}
{\begin{proof}[Proof of Claim.]}
{\qedc\renewcommand{\qed}{}\end{proof}}

\newenvironment{scproof}
{\begin{proof}[Proof of Claim.]}
{\qedsc\renewcommand{\qed}{}\end{proof}}

\newcommand{\cA}{\mathcal{A}}
\newcommand{\cB}{\mathcal{B}}

\newcommand{\cL}{\mathcal{L}}
\newcommand{\cS}{\mathcal{S}}
\newcommand{\cR}{\mathcal{R}}
\newcommand{\cT}{\mathcal{T}}
\newcommand{\cW}{\mathcal{W}}
\newcommand{\cX}{\mathcal{X}}
\newcommand{\cZ}{\mathcal{Z}}

\newcommand{\bde}{\boldsymbol{e}}

\newcommand{\bdX}{\boldsymbol{X}}
\newcommand{\xF}{\mathbf{F}}

\newcommand{\id}{\mathrm{id}}
\newcommand{\dnw}{\mathbin{\downarrow}}
\newcommand{\upw}{\mathbin{\uparrow}}
\newcommand{\res}{\mathbin{\restriction}}

\newcommand{\eps}{\varepsilon}
\newcommand{\talpha}{\tilde{\alpha}}
\newcommand{\tbeta}{\tilde{\beta}}
\newcommand{\tgamma}{\tilde{\gamma}}

\newcommand{\txi}{\tilde{\xi}}
\newcommand{\teta}{\tilde{\eta}}

\newcommand{\Reg}{\mathbf{Reg}}
\newcommand{\Ch}{\mathbf{Ch}}
\newcommand{\xV}{\mathbf{V}}

\newcommand{\zb}{0^{\mathrm{b}}}
\newcommand{\ob}{1^{\mathrm{b}}}
\DeclareMathOperator{\card}{card}

\DeclareMathOperator{\supp}{supp}
\newcommand{\ompr}{\mathbin{\ominus'}}

\newcommand{\set}[1]{\{{#1}\}}
\newcommand{\setm}[2]{\set{{#1}\mid{#2}}}

\newcommand{\famm}[2]{({#1}\mid{#2})}

\newcommand{\onto}{\twoheadrightarrow}
\newcommand{\mono}{\rightarrowtail}
\newcommand{\ol}[1]{\overline{#1}}

\newcommand{\sa}{\mathsf{a}}
\renewcommand{\sc}{\mathsf{c}}
\newcommand{\se}{\mathsf{e}}
\newcommand{\sk}{\mathsf{k}}
\newcommand{\sm}{\mathsf{m}}
\newcommand{\sx}{\mathsf{x}}
\newcommand{\sy}{\mathsf{y}}

\newcommand{\st}{\mathsf{t}}
\newcommand{\su}{\mathsf{u}}
\newcommand{\sI}{\underline{\mathsf{I}}}

\DeclareMathOperator{\Id}{Id}

\newcommand{\FF}{\mathbb{F}}
\newcommand{\LL}{\mathbb{L}}
\DeclareMathOperator{\Idemp}{Idemp}

\newcommand{\utr}{\trianglelefteq}

\author[F.~Wehrung]{Friedrich Wehrung}
\address{LMNO, CNRS UMR 6139\\
D\'epartement de Math\'ematiques, BP 5186\\
Universit\'e de Caen, Campus 2\\
14032 Caen cedex\\
France}
\email{wehrung@math.unicaen.fr}
\urladdr{http://www.math.unicaen.fr/\~{}wehrung}

\date{\today}

\begin{document}

\title[A non-coordinatizable lattice]{A non-coordinatizable sectionally complemented modular lattice with a large J\'onsson four-frame}

\begin{abstract}
A \scml~$L$ is \emph{coordinatizable} if it is isomorphic to the lattice~$\LL(R)$ of all principal right ideals of a von~Neumann regular (not necessarily unital) ring~$R$. We say that~$L$ has a \emph{large $4$-frame} if it has a homogeneous sequence $(a_0,a_1,a_2,a_3)$ such that the neutral ideal generated by~$a_0$ is~$L$. J\'onsson proved in 1962 that if~$L$ has a countable cofinal sequence and a large $4$-frame, then it is coordinatizable; whether the cofinal sequence assumption could be dispensed with was left open. We solve this problem by finding a non-coordinatizable \scml~$L$ with a large $4$-frame; it has cardinality~$\aleph_1$. Furthermore, $L$ is an ideal in a \cml~$L'$ with a spanning $5$-frame (in particular, $L'$ is coordinatizable).

Our proof uses \emph{\Ban\ functions}. A \Ban\ function on a bounded lattice~$L$ is an antitone self-map of~$L$ that picks a complement for each element of~$L$. In an earlier paper, we proved that every \emph{countable} \cml\ has a \Ban\ function. We prove that there exists a unit-regular ring~$R$ of cardinality~$\aleph_1$ and index of nilpotence~$3$ such that~$\LL(R)$ has no \Ban\ function.
\end{abstract}

\maketitle

\section{Introduction}\label{S:Intro}
\subsection{History of the problem}\label{Su:Hist}
The set~$\LL(R)$ of all principal right ideals of a (not necessarily unital) von~Neumann regular ring~$R$, ordered by inclusion, is a sublattice of the lattice of all ideals of~$L$; hence it satisfies the \emph{modular law},
 \[
 X\supseteq Z\ \Longrightarrow\ X\cap(Y+Z)=(X\cap Y)+Z\,.
 \]
(Here~$+$ denotes the addition of ideals.)
Moreover, $\LL(R)$ is \emph{sectionally complemented}, that it, for all principal right ideals~$X$ and~$Y$ such that $X\subseteq Y$, there exists a principal right ideal~$Z$ such that $X\oplus Z=Y$. A lattice is \emph{coordinatizable} if it is isomorphic to~$\LL(R)$ for some von~Neumann regular ring~$R$. In particular, every coordinatizable lattice is sectionally complemented modular. (For precise definitions we refer the reader to Section~\ref{S:Basic}.) In his monograph~\cite{Neum}, John von~Neumann proved the following result:

\begin{all}{Von Neumann's Coordinatization Theorem}
Every \cml\ that admits a spanning $n$-frame, with $n\geq4$, is coordinatizable.
\end{all}

It is not hard to find non-coordinatizable \cml s. The easiest one to describe is the lattice~$M_7$ of length two with seven atoms. Although von Neumann's original proof is very long and technical (about 150 pages), its basic idea is fairly simple: namely, assume a sufficiently rich lattice-theoretical version of a coordinate system (the spanning $n$-frame, richness being measured by the condition~$n\geq4$) to carry over the ideas in projective geometry underlying the construction of ``von Staudt's algebra of throws'' that makes it possible to go from \emph{synthetic geometry} (geometry described by incidence axioms on ``flats'') to \emph{analytic geometry} (prove statements of geometry by using coordinates and algebra), see \cite[Section~IV.5]{GLT2}. Instead of constructing (a matrix ring over) a field, von~Neumann's method yields a regular ring.

A powerful generalization of von Neumann's Coordinatization Theorem was obtained by Bjarni J\'onsson in 1960, see~\cite{Jons60}:

\begin{all}{J\'onsson's Coordinatization Theorem}
Every \cml\ $L$ that admits a large $n$-frame, with $n\geq4$ \pup{or $n\geq3$ if~$L$ is \emph{Arguesian}}, is coordinatizable.
\end{all}

There have been many simplifications, mainly due to I. Halperin~\cite{Halp61,Halp81,Halp83}, of the proof of von~Neumann's Coordinatization Theorem. A substantial simplification of the proof of J\'onsson's Coordinatization Theorem has been achieved by Christian Herrmann~\cite{Herr}---\emph{assuming the basic Coordinatization Theorem for Projective Geometries}, and thus reducing most of the complicated lattice calculations of both von~Neumann's proof and J\'onsson's proof to linear algebra. Now the Coordinatization Theorem for Projective Geometries is traditionally credited to Hilbert and to Veblen and Young, however, it is unclear whether a complete proof was published before von~Neumann's breakthrough in~\cite{Neum}. A very interesting discussion of the matter can be found in Israel Halperin's review of J\'onsson's paper~\cite{Jons60}, cf. \textbf{MR}\,0120175 (22 \#10932).

On the other hand, there is in some sense no ``Ultimate Coordinatization Theorem'' for complemented modular lattices, as the author proved that there is no first-order axiomatization for the class of all coordinatizable lattices with unit~\cite{CXCoord}.

While Von~Neumann's sufficient condition for coordinatizability requires the lattice have a unit (a spanning $n$-frame joins, by definition, to the unit of the lattice), J\'onsson's sufficient condition leaves more room for improvement. While J\'onsson assumes a unit in his above-cited Coordinatization Theorem, a large $n$-frame does not imply the existence of a unit.

And indeed, J\'onsson published in~1962 a new Coordinatization Theorem~\cite{Jons62}, assuming a large $n$-frame where $n\geq4$, where the lattice~$L$ is no longer assumed to have a unit (it is still sectionally complemented)\dots \emph{but} where the conclusion is weakened to~$L$ being isomorphic to the lattice of all finitely generated submodules of some \emph{locally projective module} over a regular ring. He also proved that if~$L$ is countable, or, more generally, has a countable cofinal sequence, then, still under the existence of a large $n$-frame, it is coordinatizable. The question whether full coordinatizability could be reached in general was left open.

In the present paper we solve the latter problem, in the negative. Our counterexample is a non-coordinatizable \scml~$L$, of cardinality~$\aleph_1$, with a large $4$-frame. Furthermore, $L$ is isomorphic to an ideal in a \cml~$L'$ with a spanning $5$-frame (in particular, $L'$ is coordinatizable).

Although our counterexample is constructed explicitly, our road to it is quite indirect. It starts with a discovery made in 1957, by Bernhard Banaschewski~\cite{Bana}, that on every vector space~$V$, over an arbitrary division ring, there exists an \emph{order-reversing} (we say \emph{antitone}) map that sends any subspace~$X$ of~$V$ to a complement of~$X$ in~$V$. Such a function was then used in order to find a simple proof of Hahn's Embedding Theorem that states that every totally ordered abelian group embeds into a generalized lexicographic power of the reals.

\subsection{\Ban\ functions on lattices and rings}\label{S:Banasch}
By analogy with \Ban's result, we define a \emph{\Ban\ function} on a bounded lattice~$L$ as an antitone self-map of~$L$ that picks a complement for each element of~$L$ (Definition~\ref{D:BanLatt}).
Hence \Ban's above-mentioned result from~\cite{Bana} states that the subspace lattice of every vector space has a \Ban\ function. This result is extended to all \emph{geometric} (not necessarily modular) lattices in Saarim\"aki and Sorjonen~\cite{SaSo}.

We proved in~\cite[Theorem~4.1]{BanFct1} that \emph{Every countable \cml\ has a \Ban\ function}. In the present paper, we construct in Proposition~\ref{P:NoBan} a unital regular ring~$S_\FF$ such that~$\LL(S_\FF)$ has no \Ban\ function. The ring~$S_\FF$ has the optimal cardinality~$\aleph_1$. Furthermore, $S_\FF$ has index~$3$ (Proposition~\ref{P:Index3}); in particular, it is unit-regular.

The construction of the ring~$S_\FF$ involves a parameter~$\FF$, which is any \emph{countable} field, and~$S_\FF$ is a ``$\FF$-algebra with quasi-inversion defined by generators and relations'' in any large enough variety. Related structures have been considered in Goodearl, Menal, and Moncasi~\cite{GoMM} and in Herrmann and Semenova~\cite{HeSe}.

\subsection{{}From non-existence of \Ban\ functions to failure of coordinatizability}\label{Su:Ban2Coord}
As we are aiming to a counterexample to the above-mentioned problem on coordinatization, we prove in Theorem~\ref{T:NoBanMeas} a stronger negative result, namely the non-existence of any ``\Ban\ measure'' on a certain increasing $\omega_1$-sequence of elements in~$L$.

A modification of this example, based on the $5\times 5$ matrix ring over~$S_\FF$, yields (Lemma~\ref{L:NonLLLiftDiagr}) an $\omega_1$-increasing chain $\vec{A}=\famm{A_\xi}{\xi<\omega_1}$ of countable \scml s, all with the same large $4$-frame, that cannot be lifted, with respect to the~$\LL$ functor, by any $\omega_1$-chain of regular rings (Lemma~\ref{L:NonLLLiftDiagr}). Our final conclusion follows from a use of a general categorical result, called the \emph{Condensate Lifting Lemma} (CLL), introduced in a paper by Pierre Gillibert and the author~\cite{GiWe}, designed to relate liftings of \emph{diagrams} and liftings of \emph{objects}. Here, CLL will turn the \emph{diagram counterexample} of Lemma~\ref{L:NonLLLiftDiagr} to the \emph{object counterexample} of Theorem~\ref{T:4/5NonCoord}. This counterexample is a so-called \emph{condensate} of the diagram~$\vec{A}$ by a suitable ``$\omega_1$-scaled Boolean algebra''. It has cardinality~$\aleph_1$ (cf. Theorem~\ref{T:4/5NonCoord}). Furthermore, it is isomorphic to an ideal of a \cml\ $L'$ with a spanning $5$-frame (so~$L'$ is coordinatizable).

\section{Basic concepts}\label{S:Basic}

\subsection{Partially ordered sets and lattices}\label{Su:PosLat}
Let~$P$ be a partially ordered set.
We denote by~$0_P$ (respectively, $1_P$) the least element (respectively, largest element) of~$P$ when they exist, also called \emph{zero} (respectively, \emph{unit}) of~$P$, and we simply write~$0$ (respectively, $1$) in case~$P$ is understood. Furthermore, we set~$P^-:=P\setminus\set{0_P}$. We set
\begin{align*}
 U\dnw X&:=\setm{u\in U}{(\exists x\in X)(u\leq x)}\,,\\
 U\upw X&:=\setm{u\in U}{(\exists x\in X)(u\geq x)}\,,
 \end{align*}
for any subsets~$U$ and~$X$ of~$P$, and we set
$U\dnw x:=U\dnw\set{x}$, $U\upw x:=U\upw\set{x}$, for any~$x\in P$. We say that~$U$ is a \emph{lower subset} of~$P$ if $U=P\dnw U$. We say that~$P$ is \emph{upward directed} if every pair of elements of~$P$ is contained in~$P\dnw x$ for some~$x\in P$. We say that~$U$ is \emph{cofinal} in~$P$ if $P\dnw U=P$. We define $p^U$ the least element of~$U\upw p$ if it exists, and we define~$p_U$ dually, for each $p\in P$. An \emph{ideal} of~$P$ is a nonempty, upward directed, lower subset of~$P$. We set
 \[
 P^{[2]}:=\setm{(x,y)\in P\times P}{x\leq y}\,.
 \]
For subsets~$X$ and~$Y$ of~$P$, let~$X<Y$ hold if $x<y$ holds for all $(x,y)\in X\times Y$. We shall also write $X<p$ (respectively, $p<X$) instead of $X<\set{p}$ (respectively, $\set{p}<X$), for each $p\in P$.
For partially ordered sets~$P$ and~$Q$, a map $f\colon P\to Q$ is \emph{isotone} (\emph{antitone}, \emph{strictly isotone}, respectively) if $x<y$ implies that $f(x)\leq f(y)$ ($f(y)\leq f(x)$, $f(x)<f(y)$, respectively), for all $x,y\in P$.

We refer to Birkhoff~\cite{Birk94} or Gr\"atzer~\cite{GLT2} for basic notions of lattice theory. We recall here a sample of needed notation, terminology, and results. In any lattice~$L$ with zero, a family $\famm{a_i}{i\in I}$ is \emph{independent} if the equality
 \[
 \bigvee\famm{a_i}{i\in X}\wedge\bigvee\famm{a_i}{i\in Y}=
 \bigvee\famm{a_i}{i\in X\cap Y}
 \]
holds for all finite subsets~$X$ and~$Y$ of~$I$. In case~$L$ is modular and~$I=\set{0,1,\dots,n-1}$ for a positive integer~$n$, this amounts to verifying that $a_k\wedge\bigvee_{i<k}a_i=0$ for each $k<n$. We denote by~$\oplus$ the operation of finite independent sum in~$L$, so $a=\bigoplus\famm{a_i}{i\in I}$ means that~$I$ is finite, $\famm{a_i}{i\in I}$ is independent, and $a=\bigvee_{i<n}a_i$. If~$L$ is modular, then~$\oplus$ is both commutative and associative in the strongest possible sense for a partial operation, see \cite[Section~II.1]{Maed58}.

A lattice~$L$ with zero is \emph{sectionally complemented} if for all $a\leq b$ in~$L$ there exists~$x\in L$ such that $b=a\oplus x$. For elements $a,x,b\in L$, let $a\sim_xb$ (respectively, $a\lesssim_xb$) hold if $a\oplus x=b\oplus x$ (respectively, $a\oplus x\leq b\oplus x$). We say that~$a$ is \emph{perspective} (respectively, \emph{subperspective}) to~$b$, in notation~$a\sim b$ (respectively, $a\lesssim b$), if there exists~$x\in L$ such that $a\sim_xb$ (respectively, $a\lesssim_xb$).
We say that~$L$ is \emph{complemented} if it has a unit and every element~$a\in L$ has a \emph{complement}, that is, an element~$x\in L$ such that $1=a\oplus x$. A bounded modular lattice is complemented if and only if it is sectionally complemented.

An ideal~$I$ of a lattice~$L$ is \emph{neutral} if $\set{I,X,Y}$ generates a distributive sublattice of~$\Id L$ for all ideals~$X$ and~$Y$ of~$L$. In case~$L$ is sectionally complemented modular, this is equivalent to the statement that every element of~$L$ perspective to some element of~$I$ belongs to~$I$. In that case, the assignment that to a congruence~$\theta$ associates the $\theta$-block of~$0$ is an isomorphism from the congruence lattice of~$L$ onto the lattice of all neutral ideals of~$L$.

An independent finite sequence $\famm{a_i}{i<n}$ in a lattice~$L$ with zero is \emph{homogeneous} if the elements~$a_i$ are pairwise perspective. An element~$x\in L$ is \emph{large} if the neutral ideal generated by~$x$ is~$L$. A family $(\famm{a_i}{0\leq i<n},\famm{c_i}{1\leq i<n})$, with $\famm{a_i}{0\leq i<n}$ independent, is a
\begin{itemize}
\item\emph{$n$-frame} if $a_0\sim_{c_i}a_i$ for each~$i$ with $1\leq i<n$;

\item\emph{large $n$-frame} if it is an $n$-frame and~$a_0$ is large.

\item\emph{spanning $n$-frame} if it is a frame, $L$ has a unit, and $1=\bigoplus_{i<n}a_i$.
\end{itemize}
In a lattice with unit, every spanning~$n$-frame is large; the converse fails for trivial examples. A \emph{large partial $n$-frame} of a \cml, as defined in J\'onsson~\cite{Jons60}, consists of a large $n$-frame as defined above, together with a finite collection of elements of~$L$ joining to the unit of~$L$ and satisfying part of the relations defining frames, so that, for instance, all of them are subperspective to~$a_0$. In particular, in a \cml, the existence of a large partial $n$-frame (as defined by J\'onsson) is equivalent to the existence of a large $n$-frame (as defined here).

\begin{definition}\label{D:n/mentire}
Let~$m$ and~$n$ be positive integers with~$m\geq n$.
A modular lattice~$L$ with zero is \emph{$n/m$-entire} if~$L$ has an ideal~$I$ and a homogeneous sequence $\famm{a_i}{i<m}$ such that, setting $a:=\bigoplus_{i<n}a_i$,
\begin{enumerate}
\item each element $x\in I$ is a join of~$m-n$ elements subperspective to~$a_0$; furthermore, $x\wedge a=0$;

\item $\setm{a\vee x}{x\in I}$ is cofinal in~$L$.
\end{enumerate}
\end{definition}

Evidently, $L$ has a spanning $n$-frame if and only if it is $n/n$-entire. Furthermore, if~$L$ is $n/m$-entire, then it has a large~$n$-frame.

\subsection{Set theory}\label{Su:SetTh}
By ``countable'' we will always mean ``at most countable''. We denote by~$\omega$ the first infinite ordinal and we identify it with the set of all non-negative integers. More generally, any ordinal~$\alpha$ is identified with the set of all ordinals smaller than~$\alpha$. Cardinals are initial ordinals. For any ordinal~$\alpha$, we denote by~$\omega_\alpha$ the $\alpha$th infinite cardinal. Following the usual set-theoretical convention, we also denote it by~$\aleph_\alpha$ whenever we wish to view it as a cardinal in the ``naive'' sense.

\v{S}anin's classical \emph{$\Delta$-Lemma} (cf. \cite[Lemma~22.6]{Jech}) is the following.

\begin{all}{$\Delta$-Lemma}
Let~$\cW$ be an uncountable collection of finite sets. Then there are an uncountable subset~$\cZ$ of~$\cW$ and a set~$Z$ \pup{the \emph{root} of~$\cZ$} such that $X\cap Y=Z$ for all distinct~$X,Y\in\cZ$.
\end{all}

We shall require the following slight strengthening of the $\Delta$-Lemma.

\begin{lemma}\label{L:Deltaplus}
Let $C$ be an uncountable subset of~$\omega_1$ and let $\famm{S_\alpha}{\alpha\in C}$ be a family of finite subsets of~$\omega_1$. Then there are an uncountable subset~$W$ of~$C$ and a set~$Z$ such that
 \[
 S_\alpha\cap S_\beta=Z\quad\text{and}\quad
 Z<S_\alpha\setminus Z<S_\beta\setminus Z\text{ for all }\alpha<\beta
 \text{ in }W\,.
 \]
\end{lemma}

\begin{proof}
By a first application of the $\Delta$-Lemma, we may assume that there exists a set~$Z$ such that $S_\alpha\cap S_\beta=Z$ for all distinct $\alpha,\beta\in C$. Put $X_\xi:=S_\xi\setminus Z$, for each $\xi\in C$.

\begin{sclaim}
For every countable $D\subset\omega_1$, there exists $\alpha\in C$ such that $D<X_\eta$ for each $\eta\in C\upw\alpha$.
\end{sclaim}

\begin{scproof}
Let~$\theta<\omega_1$ containing~$D\cup Z$.
For each $\xi\in\omega_1\setminus Z$, there exists \emph{at most} one element $f(\xi)\in C$ such that $\xi\in S_{f(\xi)}$. Any $\alpha\in C$, such that $f(\xi)<\alpha$ for each $\xi<\theta$ in the domain of~$f$, satisfies the required condition.
\end{scproof}

By applying the Claim to $D:=Z$, we get $\ol{\alpha}\in C$ such that $Z<X_\eta$ for each $\eta\in C\upw\ol{\alpha}$. Now let $\xi<\omega_1$ and suppose having constructed a strictly increasing $\xi$-sequence $\famm{\alpha_\eta}{\eta<\xi}$ in~$C\upw\ol{\alpha}$ such that $\eta<\eta'<\xi$ implies that $X_{\alpha_\eta}<X_{\alpha_{\eta'}}$. By applying the Claim to $\bigcup\famm{X_{\alpha_\eta}}{\eta<\xi}$, we obtain~$\alpha_\xi\in C$, which can be taken above both~$\ol{\alpha}$ and $\bigcup\famm{\alpha_\eta}{\eta<\xi}$, such that $X_{\alpha_\eta}<X_\zeta$ for each $\eta<\xi$ and each $\zeta\geq\alpha_\xi$. Take $W:=\setm{\alpha_\xi}{\xi<\omega_1}$.
\end{proof}

\subsection{Von Neumann regular rings}\label{Su:vNRR}
All our rings will be associative but not necessarily unital. A ring~$R$ is (von Neumann) \emph{regular} if for all $x\in R$ there exists $y\in R$ such that $xyx=x$. We shall call such an element~$y$ a \emph{quasi-inverse} of~$x$.

We shall need the following classical result (see
\cite[Theorem~1.7]{Good91}, or \cite[Section~3.6]{FrHa56} for the
general, non-unital case).

\begin{proposition}\label{P:MatReg}
For any regular ring $R$ and any positive integer $n$, the ring
$R^{n\times n}$ of all $n\times n$ matrices with entries in $R$ is
regular.
\end{proposition}

For any regular ring~$R$, we set $\LL(R):=\setm{xR}{x\in R}$. If~$y$ is a quasi-inverse of~$x$, then $xR=xyR$ and~$xy$ is idempotent, thus $\LL(R)=\setm{eR}{e\in R\text{ idempotent}}$. It is well known that~$\LL(R)$ is a sectionally complemented sublattice of the (modular) lattice of all right ideals of~$R$ (cf. \cite[Section~3.2]{FrHa54}). The proof implies that~$\LL$ defines a \emph{functor} from the category of all regular rings with ring homomorphisms to the category of \scml s with $0$-lattice homomorphisms (cf. Micol~\cite{Micol} for details). This functor preserves directed colimits.

\begin{lemma}[folklore]\label{L:L(R)unital}
A regular ring~$R$ is unital if and only if~$\LL(R)$ has a largest element.
\end{lemma}

\begin{proof}
We prove the non-trivial direction. Let~$e\in R$ idempotent such that~$eR$ is the largest element of~$\LL(R)$. For each~$x\in R$ with quasi-inverse~$y$, observe that $x=xyx\in xR$, thus, as~$xR\subseteq eR$ and by the idempotence of~$e$, we get $x=ex$. Let~$y$ be a quasi-inverse of~$x-xe$. {}From~$y=ey$ it follows that $xy-xey=0$, thus
 \[
 x-xe=(x-xe)y(x-xe)=(xy-xey)(x-xe)=0\,,
 \]
so $x=xe$. Therefore, $e$ is the unit of~$R$.
\end{proof}

Denote by $\Idemp R$ the set of all idempotent elements in a ring~$R$. Define the \emph{orthogonal sum} in~$\Idemp R$ by
 \[
 a=\bigoplus_{i<n}a_i\ \Longleftrightarrow\ 
 \Bigl(a=\sum_{i<n}a_i\text{ and }
 a_ia_j=0\text{ for all distinct }i,j<n\Bigr)\,.
 \]
For idempotents $a$ and $b$ in a ring~$R$, let $a\utr b$ hold if $a=ab=ba$; equivalently, there exists an idempotent~$x$ such that $a\oplus x=b$; and equivalently, $a\in bRb$. 

We shall need the following well known (and easy) result.

\begin{lemma}[folklore]\label{L:DecompId}
Let $A$ and $B$ be right ideals in a ring~$R$ and let~$e$ be an idempotent element of~$R$. If $eR=A\oplus B$, then there exists a unique pair $(a,b)\in A\times B$ such that $e=a+b$. Furthermore, both~$a$ and~$b$ are idempotent, $e=a\oplus b$, $A=aR$, and $B=bR$.
\end{lemma}

\subsection{Category theory}\label{Su:CatTh}
For a partially ordered set~$I$ and a category~$\cA$, an \emph{$I$-indexed diagram from~$\cA$} is a system $\famm{A_i,f_i^j}{i\leq j\text{ in }I}$, where all~$A_i$ are objects in~$\cA$, $f_i^j\colon A_i\to A_j$ in~$\cA$, and $f_i^i=\id_{A_i}$ together with $f_i^k=f_j^k\circ f_i^j$ for $i\leq j\leq k$ in~$I$. Such an object can of course be identified with a functor from~$I$, viewed as a category the usual way, to~$\cA$. If~$\cB$ is a category, $\Phi\colon\cA\to\cB$ is a functor, and~$\vec{B}$ is an $I$-indexed diagram from~$\cB$, we say that an $I$-indexed diagram~$\vec{A}$ from~$\cA$ \emph{lifts~$\vec{B}$ with respect to~$\Phi$} if there is a natural equivalence from~$\Phi\vec{A}$ to~$\vec{B}$ (in notation $\Phi\vec{A}\cong\vec{B}$).

\section{\Ban\ functions on lattices and rings}\label{S:BanLattRing}

In the present section we recall some definitions and results from~\cite{BanFct1}.

\begin{definition}\label{D:BanLatt}
Let $X$ be a subset in a bounded lattice~$L$. A \emph{partial \Ban\ function on~$X$ in~$L$} is an antitone map $f\colon X\to L$ such that $x\oplus f(x)=1$ for each $x\in X$. In case~$X=L$, we say that~$f$ is a \emph{\Ban\ function on~$L$}.
\end{definition}

\begin{definition}\label{D:BanRing}
Let $X$ be a subset in a ring~$R$. A \emph{partial \Ban\ function on~$X$ in~$R$} is a mapping $\eps\colon X\to\Idemp R$ such that
\begin{enumerate}
\item $xR=\eps(x)R$ for each $x\in X$.

\item $xR\subseteq yR$ implies that $\eps(x)\utr\eps(y)$, for all $x,y\in X$.
\end{enumerate}
In case $X=R$ we say that~$f$ is a \emph{\Ban\ function on~$R$}.
\end{definition}

In the context of Definition~\ref{D:BanRing}, we put
 \begin{equation}\label{Eq:DefnLR(X)}
 \LL_R(X):=\setm{xR}{x\in X}\,.
 \end{equation}
 
We proved the following result in \cite[Lemma~3.5]{BanFct1}.

\begin{lemma}\label{L:BanLattRing}
Let~$R$ be a unital regular ring and let $X\subseteq R$. Then the following are equivalent:
\begin{enumerate}
\item There exists a partial \Ban\ function on~$\LL_R(X)$ in~$\LL(R)$.

\item There exists a partial \Ban\ function on~$X$ in~$R$.
\end{enumerate}
\end{lemma}

\section{A coordinatizable \cml\ without a \Ban\ function}\label{S:CoordAl1}

For a field~$\FF$, we consider the similarity type~$\Sigma_\FF=(0,1,-,\cdot,',\famm{h_\lambda}{\lambda\in\FF})$ that consists of two symbols of constant~$0$ and~$1$, two binary operation symbols~$-$ (difference) and~$\cdot$ (multiplication), one unary operation symbol~$'$ (quasi-inversion), and a family of unary operations~$h_\lambda$, for $\lambda\in\FF$ (left multiplications by the elements in~$\FF$). We consider the variety~$\Reg_\FF$ of all unital $\FF$-algebras with a distinguished operation~$x\mapsto x'$ in which the identity $xx'x=x$ holds (i.e., $x\mapsto x'$ is a quasi-inversion). We shall call~$\Reg_\FF$ the \emph{variety of all $\FF$-algebras with quasi-inversion}. Of course, all the ring reducts of the structures in~$\Reg_\FF$ are regular, and the reducts of such structures to the subtype $\Sigma:=(0,-,\cdot,')$ are \emph{regular rings with quasi-inversion}.

Until Proposition~\ref{P:cRonetoone} we shall fix a variety (i.e., the class of all the structures satisfying a given set of identities) $\xV$ of~$\Sigma_\FF$-structures contained in~$\Reg_\FF$. By \cite[Theorem~V.11.2.4]{Malc}, it is possible to construct ``objects defined by generators and relations'' in any (quasi-)variety.

\begin{definition}\label{D:cR(A)}
For any (possibly empty) chain~$\Lambda$, we shall denote by $\cR_\xV(\Lambda)$ the $\xV$-object defined by generators $\talpha$, for $\alpha\in\Lambda$, and the relations
 \begin{equation}\label{Eq:DefRelPhi}
 \talpha=\tbeta\cdot\talpha\,,\quad\text{for all }\alpha\leq\beta\text{ in }
 \Lambda\,.
 \end{equation}
We shall write $\talpha^\Lambda$ instead of~$\talpha$ in case~$\Lambda$ needs to be specified.
\end{definition}

Observe, in particular, that the $(0,1,-,\cdot,\famm{h_\lambda}{\lambda\in\FF})$-reduct of~$\cR_\xV(\Lambda)$ is a regular $\FF$-algebra.

For a chain~$\Lambda$, denote by $\Lambda\sqcup\set{\zb,\ob}$ the chain obtained by adjoining to~$\Lambda$ a new smallest element~$\zb$ and a new largest element~$\ob$. Likewise, define $\Lambda\sqcup\set{\zb}$ and $\Lambda\sqcup\set{\ob}$. We extend the meaning of $\talpha$, for $\alpha\in\Lambda\sqcup\set{\zb,\ob}$, by setting 
 \begin{equation}\label{Eq:projzbob}
 \widetilde{\zb}=0\text{ and }\widetilde{\ob}=1\,.
 \end{equation}
The equations~\eqref{Eq:DefRelPhi} are still satisfied for all $\alpha\leq\beta$ in $\Lambda\sqcup\set{\zb,\ob}$.

Denote by~$\Ch$ the category whose objects are all the (possibly empty) chains and where, for chains~$A$ and~$B$, a \emph{morphism} from~$A$ to~$B$ is an isotone map from $A\sqcup\set{\zb,\ob}$ to $B\sqcup\set{\zb,\ob}$ fixing both~$\zb$ and~$\ob$. In particular, we identify every isotone map from~$A$ to~$B$ with its extension that fixes both~$\zb$ and~$\ob$. This occurs, in particular, in case~$A$ is a subchain of~$B$ and $f:=e_A^B$ is the inclusion map from~$A$ into~$B$; in this case, we put $\bde_A^B:=\cR_\xV(e_A^B)$, the canonical $\Sigma_\FF$-morphism from~$\cR_\xV(A)$ to~$\cR_\xV(B)$.

Every morphism $f\colon A\to B$ in~$\Ch$ induces a (unique) $\Sigma_\FF$-homomorphism\linebreak $\cR_\xV(f)\colon\cR_\xV(A)\to\cR_\xV(B)$ by the rule
 \begin{equation}\label{Eq:DefcR(f)}
 \cR_\xV(f)(\talpha^A)=
 \widetilde{f(\alpha)}^B\,,
 \quad\text{for each }\alpha\in A
 \end{equation}
(use \eqref{Eq:DefRelPhi} and \eqref{Eq:projzbob}).
The assignments $\Lambda\mapsto\cR_\xV(\Lambda)$, $f\mapsto\cR_\xV(f)$ define a functor from~$\Ch$ to~$\xV$. For a chain~$\Lambda$ and an element $x\in\cR_\xV(\Lambda)$, there are a $\Sigma_\FF$-term~$\st$ and finitely many elements~$\xi_1$, \dots, $\xi_n\in\Lambda$ such that
 \begin{equation}\label{Eq:Reprxxi}
 x=\st(\txi_1,\dots,\txi_n)\,.
 \end{equation}
in~$\cR_\xV(\Lambda)$.
Any subset of~$\Lambda$ containing $\set{\xi_1,\dots,\xi_n}$ is called a \emph{support} of~$x$. In particular, every element of~$\cR_\xV(\Lambda)$ has a finite support, and a subset~$S$ is a support of~$x$ if and only if~$x$ belongs to the range of~$\bde_S^\Lambda$.

\begin{lemma}\label{L:ImSupp}
Let $A$ and $B$ be chains and let~$f$ be a morphism from~$A$ to~$B$ in~$\Ch$. Let~$x\in\cR_\xV(A)$ and let~$S$ be a support of~$x$. Then~$f(S)\setminus\set{\zb,\ob}$ is a support of~$\cR_\xV(f)(x)$.
\end{lemma}

\begin{proof}
There is a representation of~$x$ as in~\eqref{Eq:Reprxxi} in~$\cR_\xV(A)$, with $\xi_1,\dots,\xi_n\in S$. As~$\cR_\xV(f)$ is a $\Sigma_\FF$-homomorphism, we obtain
 \[
 \cR_\xV(f)(x)=\st\bigl(\widetilde{f(\xi_1)},\dots,\widetilde{f(\xi_n)}\bigr)
 \quad\text{in }\cR_\xV(B)\,.
 \]
As $\widetilde{f(\xi_i)}$ belongs to $f(S)\cup\set{0,1}$ for each~$i$ and both elements~$0$ and~$1$ of~$\cR_\xV(B)$ are interpretations of symbols of constant, the conclusion follows.
\end{proof}

The following result implies immediately that all maps $\bde_A^B\colon\cR_\xV(A)\to\cR_\xV(B)$, for~$A$ a subchain of a chain~$B$, are $\Sigma_\FF$-embeddings.

\begin{proposition}\label{P:cRonetoone}
Let $A$ and $B$ be chains and let $f\colon A\to B$ be an isotone map. If~$f$ is one-to-one, then so is~$\cR_\xV(f)$.
\end{proposition}

\begin{proof}
It suffices to prove that $\cR_\xV(f)(x)=0$ implies that $x=0$, for each $x\in\cR_\xV(A)$. There is a representation of~$x$ as in~\eqref{Eq:Reprxxi} in~$\cR_\xV(A)$. Put $S:=\set{\xi_1,\dots,\xi_n}$ and $u:=\st(\txi_1^S,\dots,\txi_n^S)$. Let $g\colon B\to S\sqcup\set{\zb}$ be the map defined by the rule
 \[
 g(\beta):=\begin{cases}
 \text{largest }\xi\in S\text{ such that }f(\xi)\leq\beta\,,
 &\text{if such a }\xi\text{ exists},\\
 \zb\,,&\text{otherwise},
 \end{cases}\quad\text{for each }\beta\in B\,.
 \]
It is obvious that $g$ is isotone. Furthermore, as $f$ is one-to-one and isotone, we obtain $g\circ f\circ e_S^A=\id_S$, so $\cR_\xV(g)\circ\cR_\xV(f)\circ\bde_S^A=\id_{\cR_\xV(S)}$, and so, using the equality $\cR_\xV(f)(x)=0$,
 \[
 u=\cR_\xV(g)\circ\cR_\xV(f)\circ\bde_S^A(u)=\cR_\xV(g)\circ\cR_\xV(f)(x)=0\,,
 \]
and therefore $x=\bde_S^A(u)=0$.
\end{proof}

Now we shall put more conditions on the variety~$\xV$ of $\FF$-algebras with quasi-inversion. We fix a \emph{countable} field~$\FF$, and we consider the following elements in the matrix ring~$\FF^{3\times 3}$:
 \[
 A:=\begin{pmatrix}1&0&0\\ 0&0&0\\ 0&0&0\end{pmatrix}\,,\quad
 B:=\begin{pmatrix}1&0&1\\ 0&1&0\\ 0&0&0\end{pmatrix}\,,\quad
 I:=\begin{pmatrix}1&0&0\\ 0&1&0\\ 0&0&1\end{pmatrix}\,.
 \]
Observe that $A^2=A$, $B^2=B$, and $A=BA\neq AB$.

Denote by $\FF[M]$ the $\FF$-subalgebra of $\FF^{3\times 3}$ generated by~$\set{M}$, for any $M\in\FF^{3\times 3}$. In particular, both maps from~$\FF\times\FF$ to~$\FF^{3\times 3}$ defined by $(x,y)\mapsto xA+y(I-A)$ and $(x,y)\mapsto xB+y(I-B)$ are isomorphisms of $\FF$-algebras onto~$\FF[A]$ and~$\FF[B]$, respectively, and $\FF[A]\cap\FF[B]=\FF\cdot I$. For each $X\in\FF^{3\times 3}$, let~$X'$ be a quasi-inverse of~$X$ in the smallest member of $\set{\FF\cdot I,\FF[A],\FF[B],\FF^{3\times 3}}$ containing~$X$ as an element. Endowing each of the algebras $\FF\cdot I$, $\FF[A]$, $\FF[B]$, and $\FF^{3\times 3}$ with this quasi-inversion, we obtain a commutative diagram in~$\Reg_\FF$, represented in Figure~\ref{Fig:CommSq}.
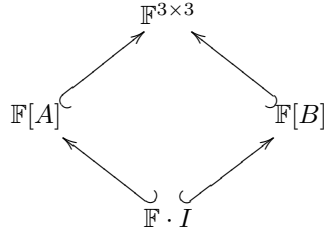
\begin{figure}[htb]
 \[
 \xymatrix{
 & \FF^{3\times 3} & \\
 \FF[A]\ar@{_(->}[ur] & & \FF[B]\ar@{^(->}[ul]\\
 & \FF\cdot I\ar@{_(->}[ul]\ar@{^(->}[ur] & }
 \]
\caption{A commutative square in the variety $\Reg_\FF$}
\label{Fig:CommSq}
\end{figure}
We denote by~$R_\FF$ the $\FF$-algebra with quasi-inversion on $\FF^{3\times 3}$ just constructed, and we denote by~$\xV_\FF$ the variety of $\FF$-algebras with quasi-inversion generated by~$R_\FF$.

\begin{proposition}\label{P:NoBan}
Let~$\xV$ be any variety of $\FF$-algebras with quasi-inversion such that $R_\FF\in\xV$. Then the following statements hold:
\begin{enumerate}
\item There exists no partial \Ban\ function on $\setm{\txi}{\xi<\omega_1}$ in the \pup{unital, regular} ring~$\cR_\xV(\omega_1)$. In particular, there is no \Ban\ function on the ring $\cR_\xV(\omega_1)$.

\item There exists no partial \Ban\ function on $\setm{\txi\cdot\cR_\xV(\omega_1)}{\xi<\omega_1}$ in the \pup{complemented, modular} lattice $\LL(\cR_\xV(\omega_1))$. In particular, there is no \Ban\ function on the lattice $\LL(\cR_\xV(\omega_1))$.
\end{enumerate}
\end{proposition}

\begin{proof}
A direct application of Lemma~\ref{L:BanLattRing} shows that it is sufficient to establish the result of the first sentence of~(i).

Set $X:=\setm{\txi}{\xi<\omega_1}$ and suppose that there exists a partial \Ban\ function $\rho\colon X\to\Idemp\cR_\xV(\Lambda)$. For each $\xi<\omega_1$, there exists $u_\xi\in\cR_\xV(\omega_1)$ such that
 \begin{equation}\label{Eq:bunchxiuxi}
 \txi=\txi\cdot u_\xi\cdot\txi\text{ and }
 \rho(\txi)=\txi\cdot u_\xi\quad\text{in }\cR_\xV(\Lambda)\,.
 \end{equation} 
Pick a finite support~$S_\xi$ of~$u_\xi$ containing~$\set{\xi}$, for each $\xi<\omega_1$. By Lemma~\ref{L:Deltaplus}, there are a (finite) set~$Z$ and an uncountable subset~$W$ of~$\omega_1$ such that
 \begin{equation}\label{Eq:XalphaIncr}
 S_\xi\cap S_\eta=Z\text{ and }
 Z<S_\xi\setminus Z<S_\eta\setminus Z
 \text{ for all }\xi<\eta\text{ in }W\,.
 \end{equation}
Put $S'_\xi:=S_\xi\setminus Z$, for each $\xi\in W$.
We define a map $f\colon\omega_1\to W\sqcup\set{\zb}$ by the rule
 \[
 f(\alpha):=\begin{cases}
 \text{least }\xi\in W\text{ such that }\alpha\in\omega_1\dnw S'_\xi\,,&
 \text{if }\alpha\in\omega_1\upw S'_0\,,\\
 \zb\,,&\text{otherwise},
 \end{cases}
 \quad\text{for each }\alpha<\omega_1\,.
 \]
The precaution to separate the case where $\alpha\in\omega_1\dnw S'_\xi$ is put there in order to ensure, using~\eqref{Eq:XalphaIncr}, that $f(\alpha)=\zb$ for each $\alpha\in Z$. Observe that~$f$ is isotone and (using~\eqref{Eq:XalphaIncr} again) that the restriction of~$f$ to~$S'_\xi$ is the constant map with value~$\xi$, for each~$\xi\in W$. In particular, $f\res_W=\id_W$.

Set $v_\xi:=\cR_\xV(f)(u_\xi)$ and $e_\xi:=\cR_\xV(f)(\rho(\txi))$, for each $\xi\in W$. By applying the morphism~$\cR_\xV(f)$ to~\eqref{Eq:bunchxiuxi}, we thus obtain that
 \begin{equation}\label{Eq:bunchinW}
 \txi=\txi\cdot v_\xi\cdot\txi\text{ and }
 e_\xi=\txi\cdot v_\xi\quad\text{in }\cR_\xV(W)\,,\quad
 \text{for each }\xi\in W\,.
 \end{equation}
Furthermore, by applying $\cR_\xV(f)$ to the relation $\rho(\txi)\utr\rho(\teta)$, we obtain the system of relations
 \begin{equation}\label{Eq:alphbetutrW}
 e_\xi\utr e_\eta\text{ in }\cR_\xV(W)\,,\quad\text{for all }
 \xi\leq\eta\text{ in }W\,.
 \end{equation}
Furthermore, as~$u_\xi$ has support~$S_\xi$ and $f(S_\xi)=f(Z)\cup f(S'_\xi)\subseteq\set{\zb,\xi}$, it follows from Lemma~\ref{L:ImSupp} that~$\set{\xi}$ is a support of~$v_\xi$, so $v_\xi=\st_\xi(\txi)$ for some term~$\st_\xi$ of~$\Sigma_\FF$. As~$\FF$ is countable, there are only countably many terms in~$\Sigma_\FF$, thus, as~$W$ is uncountable, we may trim~$W$ further in order to ensure that there exists a term~$\st$ of~$\Sigma_\FF$ such that $\st_\xi=\st$ for each $\xi\in W$. Therefore, we have obtained that
 \begin{equation}\label{Eq:valphatalpha}
 v_\xi=\st(\txi)\quad\text{in }\cR_\xV(W)\,,\quad\text{for each }\xi\in W\,.
 \end{equation}
Denote by~$\se$ the term of $\Sigma_\FF$ defined by $\se(\sx)=\sx\cdot\st(\sx)$. In particular, from~\eqref{Eq:bunchinW} and~\eqref{Eq:valphatalpha} it follows that $e_\xi=\se(\txi)$ for each $\alpha\in W$.

{}From now on until the end of the proof, we shall fix $\alpha<\beta$ in~$W$. As the $\FF$-algebra with quasi-inversion~$R_\FF$ (with underlying ring $\FF^{3\times 3}$) belongs to the variety~$\xV$, as both~$A$ and~$B$ are idempotent with $A=BA$, and by the definition of~$\cR_\xV(W)$, there exists a unique $\Sigma_\FF$-homomorphism $\varphi\colon\cR_\xV(W)\to R_\FF$ such that
 \[
 \varphi(\txi)=\begin{cases}
 A\,,&\text{if }\xi\leq\alpha\,,\\
 B\,,&\text{otherwise},
 \end{cases}\quad\text{for each }\xi\in W\,.
 \]
By applying the homomorphism~$\varphi$ to the equation $v_\alpha=\st(\talpha)$, we obtain that $\varphi(v_\alpha)=\st(A)$ belongs to~$\FF[A]$ (because~$\FF[A]$ is a $\Sigma_\FF$-substructure of~$R_\FF$). Similarly, $\varphi(v_\beta)=\st(B)$ belongs to~$\FF[B]$. Using~\eqref{Eq:bunchinW}, it follows that 
 \begin{equation}\label{Eq:At(A)A}
 \varphi(e_\alpha)=\se(A)\,,\quad\varphi(e_\beta)=\se(B)\,,\quad
 A=A\cdot \st(A)\cdot A\,,\quad B=B\cdot \st(B)\cdot B\,.
 \end{equation}
{}From the third equation in~\eqref{Eq:At(A)A} it follows that $A\cdot\FF[A]=(A\cdot \st(A))\cdot\FF[A]=\se(A)\cdot\FF[A]$. As the only non-trivial idempotent elements of~$\FF[A]$ are~$A$ and~$I-A$, this leaves the only possibility~$\se(A)=A$. Similarly, $\se(B)=B$.

However, by applying the homomorphism~$\varphi$ to the relation~\eqref{Eq:alphbetutrW}, we obtain that $\se(A)\utr\se(B)$ in~$R_\FF$ (\emph{it is here that we really need the countability of~$\FF$, for we need $\st_\alpha=\st_\beta$}!), so $A\utr B$. In particular, $A=AB$, a contradiction.
\end{proof}

Proposition~\ref{P:NoBan} applies in particular to the case where~$\xV$ is the variety~$\xV_\FF$ generated by the algebra~$R_\FF$, that is, the class of all $\Sigma_\FF$-structures satisfying all the identities (in the similarity type~$\Sigma_\FF$) satisfied by~$R_\FF$.

The following result shows an additional property of the algebras~$\cR_\FF(\Lambda):=\cR_{\xV_\FF}(\Lambda)$. Recall that the \emph{index of nilpotence} of a nilpotent element~$a$ in a ring~$T$ is the least positive integer~$n$ such that $a^n=0$, and the \emph{index of~$T$} is the supremum of the indices of all elements of~$T$.

\begin{proposition}\label{P:Index3}
Every member of the variety~$\xV_\FF$ has index at most~$3$. In particular, the algebra~$\cR_\FF(\Lambda)$ has index at most~$3$, for every chain~$\Lambda$.
\end{proposition}

\begin{proof}
By Birkhoff's \textbf{HSP} Theorem in Universal Algebra (see, for example, Theorems~9.5 and~11.9 in Burris and Sankappanavar~\cite{BuSa}), every member~$T$ of~$\xV_\FF$ is a $\Sigma_\FF$-homomorphic image of a $\Sigma_\FF$-substructure of a power of~$R_\FF$. As the underlying $\FF$-algebra of~$R_\FF$ is~$\FF^{3\times 3}$, it has index~$3$ (cf. \cite[Theorem~7.2]{Good91}), thus so does every power of~$R_\FF$, and thus also every subalgebra of every power of~$R_\FF$. As taking homomorphic images does not increase the index of regular rings (cf. \cite[Proposition~7.7]{Good91}), $T$ has index at most~$3$.
\end{proof}

\begin{remark}\label{Rk:NoBan1}
It follows from Proposition~\ref{P:Index3} that $\cR_\FF(\omega_1)$ has index at most~$3$ (it is not hard to see that it is exactly~$3$). In particular, by \cite[Corollary~7.11]{Good91}, $\cR_\FF(\omega_1)$ is unit-regular.

If~$\FF$ is finite, then more can be said. Set $R:=R_\FF$ for brevity. It follows from one of the proofs of Birkhoff's \textbf{HSP} Theorem that the free algebra~$F_n$ on~$n$ generators in the variety~$\xV_\FF$ is isomorphic to the $\Sigma_\FF$-substructure of~$R^{R^n}$ generated by the~$n$ canonical projections from~$R^n$ onto~$R$. In particular, $F_n$ is finite. It follows that the $\FF$-algebra with quasi-inversion~$\cR_\FF(\Lambda)$ is \emph{locally finite}.

To summarize, we have obtained that \emph{If~$\FF$ is a finite field, then $\cR_\FF(\omega_1)$ is a locally finite regular~$\FF$-algebra with index~$3$, but without a \Ban\ function}.
\end{remark}

\begin{remark}\label{Rk:IncromCh}
Part (a) of \cite[Proposition~2.13]{Good91} implies that for every increasing sequence (indexed by the non-negative integers) $\famm{I_n}{n<\omega}$ of principal right ideals in a unital regular ring~$R$, there exists a $\utr$-increasing sequence $\famm{e_n}{n<\omega}$ of idempotents of~$R$ such that $I_n=e_nR$ for each~$n<\omega$. The origin of this argument can be traced back to Kaplansky's proof that every countably generated right ideal in a regular ring is projective \cite[Lemma~1]{Kapl58}.

Proposition~\ref{P:NoBan} implies that the result above cannot be extended to $\omega_1$-sequences of principal right ideals, even if the ring~$R$ has bounded index by Proposition~\ref{P:Index3}.

Observe that Kaplansky finds in~\cite{Kapl58} a non-projective (uncountable) right ideal in a regular ring. Another example, suggested to the author by Luca Giudici, runs as follows. Let~$X$ be a locally compact, Hausdorff, non paracompact zero-dimensional space. A classical example of such a space is given by the closed subspace of Dieudonn\'e's long ray consisting of the first uncountable ordinal~$\omega_1$ endowed with its order topology (all intervals of the form either~$\omega_1\dnw\alpha$ or $\omega_1\upw\alpha$, for $\alpha<\omega_1$, form a basis of closed sets of the topology). Now let~$Y$ be the one-point compactification of~$X$. Denote by~$B$ the Boolean algebra of all clopen subsets of~$Y$, and by~$I$ the ideal of~$B$ consisting of all the clopen subsets of~$X$. Then~$B$ is a commutative regular ring and~$I$ is a non-projective ideal of~$B$ (cf. Bkouche~\cite{Bkou70}, Finney and Rotman~\cite{FiRo70}). In the particular case where~$X$ is the example above, $I$ is the union of the increasing chain of principal ideals corresponding to the intervals $[0,\alpha]$, for $\alpha<\omega_1$.

However, we do not know any relation, beyond the formal analogy outlined above, between projectivity of ideals and existence of \Ban\ functions. In particular, while Kaplansky's construction in~\cite{Kapl58} is given as an algebra over \emph{any} field~$\FF$, the construction of our counterexample in Section~\ref{S:CoordAl1} requires~$\FF$ be \emph{countable}. Moreover, in Giudici's example above, the identity function on~$B$ is a \Ban\ function on (the ring)~$B$.
\end{remark}

\section{\Ban\ measures on subsets of lattices with zero}\label{S:BanMeas}

In order to reach our final coordinatization failure result (Theorem~\ref{T:4/5NonCoord}) we need the following variant of \Ban\ functions, introduced in~\cite[Definition~5.5]{BanFct1}.

\begin{definition}\label{D:BanMeas}
Let $X$ be a subset in a lattice~$L$ with zero. A \emph{$L$-valued \Ban\ measure on~$X$} is a map $\ominus\colon X^{[2]}\to L$, $(x,y)\mapsto y\ominus x$, isotone in~$y$ and antitone in~$x$, such that $y=x\oplus(y\ominus x)$ for all $x\leq y$ in~$X$.
\end{definition}

The following lemma gives us an equivalent definition in case~$L$ is \emph{modular}.

\begin{lemma}\label{L:BanMeasMod}
Let $X$ be a subset in a modular lattice~$L$ with zero. Then a map $\ominus\colon X^{[2]}\to L$ is a \Ban\ measure if and only if
 \begin{equation}\label{Eq:EquivBanMeas}
 y=x\oplus(y\ominus x)\quad\text{and}\quad
 z\ominus x=(z\ominus y)\oplus(y\ominus x)\,,\quad
 \text{for all }x\leq y\leq z\text{ in }X\,.
 \end{equation}
Furthermore, if this holds, then
 \begin{equation}\label{Eq:Measxyfromz}
 y\ominus x=y\wedge(z\ominus x)\,,\quad\text{for all }
 x\leq y\leq z\text{ in }X\,. 
 \end{equation}
\end{lemma}

\begin{proof}
Condition \eqref{Eq:EquivBanMeas} trivially implies that~$\ominus$ is a \Ban\ measure on~$X$. Conversely, assume that~$\ominus$ is a \Ban\ measure on~$X$, and let $x\leq y\leq z$ in~$X$. The equality $y=x\oplus(y\ominus x)$ follows from the definition of a \Ban\ measure. As, in addition, $z=y\oplus(z\ominus y)$ and from the associativity of the partial operation~$\oplus$ (which follows from the modularity of~$L$), it follows that $z=x\oplus u$ where $u:=(z\ominus y)\oplus(y\ominus x)$. Hence both~$u$ and $z\ominus x$ are sectional complements of~$x$ in~$z$ with $u\leq z\ominus x$, whence, by the modularity of~$L$, $u=z\ominus x$. This concludes the proof of the first equivalence.

Now assume that~$\ominus$ is a \Ban\ measure on~$X$, let $x\leq y\leq z$ in~$X$, and set $v:=y\wedge(z\ominus x)$. Trivially, $x\wedge v=0$. Furthermore, as $x\leq y$ and by the modularity of~$L$,
 \[
 x\vee v=y\wedge\bigl(x\vee(z\ominus x)\bigr)=y\wedge z=y\,.
 \]
Therefore, $x\oplus(y\ominus x)=y=x\oplus v$, thus, as $y\ominus x\leq v$ and~$L$ is modular, $v=y\ominus x$.
\end{proof}

\begin{lemma}\label{L:Transfer}
Let $L$ be a modular lattice with zero, let $e,b\in L$ such that $e\oplus b=1$, and let $X\subseteq L\dnw b$. If there exists an $L$-valued \Ban\ function on~$e\oplus X:=\setm{e\oplus x}{x\in X}$, then there exists a $(L\dnw b)$-valued \Ban\ function on~$X$.
\end{lemma}

\begin{proof}
By assumption, there exists an $L$-valued \Ban\ measure $\ominus$ on $e\oplus X$. We set
 \[
 y\ompr x:=
 b\wedge\bigl[e\vee\bigl((e\oplus y)\ominus(e\oplus x)\bigr)\bigr]\,,
 \quad\text{for all }x\leq y\text{ in }X\,.
 \]
Clearly, the map $\ompr$ thus defined is $(L\dnw b)$-valued, and  isotone in~$y$ while antitone in~$x$. For all $x\leq y$ in~$X$, it follows from the equation $e\oplus y=e\oplus x\oplus\bigl((e\oplus y)\ominus(e\oplus x)\bigr)$ and the modularity of~$L$ that
 \[
 x\wedge\bigl[e\vee\bigl((e\oplus y)\ominus(e\oplus x)\bigr)\bigr]=0\,,
 \]
so, as $x\leq b$, we get $x\wedge(y\ompr x)=0$. On the other hand,
 \begin{align*}
 x\vee(y\ompr x)&=
 b\wedge\bigl[x\vee e\vee\bigl((e\oplus y)\ominus(e\oplus x)\bigr)\bigr]
 &&(\text{because }x\leq b\text{ and }L\text{ is modular})\\
 &=b\wedge(e\vee y)\\
 &=(b\wedge e)\vee y&&
 (\text{because }y\leq b\text{ and }L\text{ is modular})\\
 &=y\,,
 \end{align*}
so $x\oplus(y\ompr x)=y$.
\end{proof}

\section{An $\omega_1$-sequence without a \Ban\ measure}\label{S:NonBanMeas}

Throughout this section we shall use the notation of Section~\ref{S:CoordAl1}. A term~$\st$ of a similarity type containing $\Sigma:=(0,-,\cdot,')$ is \emph{strongly idempotent} if either $\st=\su\cdot\su'$ or $\st=\su'\cdot\su$ for some term~$\su$ of~$\Sigma$. We define strongly idempotent terms~$\sk$ and~$\sm$ of~$\Sigma$ by
 \begin{align}
 \sk(\sx,\sy)&:=(\sy\sy'-\sx\sx'\sy\sy')'\cdot(\sy\sy'-\sx\sx'\sy\sy')\,,
 \label{Eq:Defse}\\
 \sm(\sx,\sy)&:=\bigl(\sy\sy'-\sy\sy'\sk(\sx,\sy)\bigr)\cdot
 \bigl(\sy\sy'-\sy\sy'\sk(\sx,\sy)\bigr)'\,.\label{Eq:Defsm}
 \end{align}
We shall need the following lemma, that follows immediately from the trivial fact that $xx'R=xR$ for any element~$x$ with quasi-inverse~$x'$ in a regular ring~$R$, together with \cite[Section~3.2]{FrHa54}.

\begin{lemma}\label{L:Pmeet}
The equality $xR\cap yR=\sm(x,y)R$ holds, for any elements~$x$ and~$y$ in a regular ring~$R$ with quasi-inversion.
\end{lemma} 

Until the statement of Theorem~\ref{T:NoBanMeas} we shall fix a countable field~$\FF$ and a variety~$\xV$ of regular $\FF$-algebras with quasi-inversion. We shall denote by $\cL_\xV:=\LL\circ\cR_\xV$ the composite functor (from~$\Ch$ to the category of all sectionally complemented modular lattices with $0$-lattice homomorphisms).

A subset $S$ in a chain~$\Lambda$ is a \emph{support} of an element $I\in\cL_\xV(\Lambda)$ if~$I$ belongs to the range of $\cL_\xV(e_S^\Lambda)$. Equivalently, $I=x\cdot\cR_\xV(\Lambda)$ for some $x\in\cR_\xV(\Lambda)$ with support~$S$.

\begin{lemma}\label{L:ToSmSupp}
Let $\Lambda$ be a chain, let $I\in\cL_\xV(\Lambda)$, let $X\subseteq\Lambda$, and let $\xi\in\Lambda$. If both~$X$ and~$\Lambda\setminus\set{\xi}$ support~$I$, then $X\setminus\set{\xi}$ supports~$I$.
\end{lemma}

\begin{proof}
As some finite subset of~$X$ is a support of~$I$, we may assume that~$X$ is finite. Moreover, the conclusion is trivial in case $\xi\notin X$, so we may assume that $\xi\in X$. Let $f\colon\Lambda\to\Lambda\sqcup\set{\zb,\ob}$ defined by
 \[
 f(\eta):=\begin{cases}
 \xi\,,&\text{if }\eta=\xi\,,\\
 \eta^X\,,&\text{if }\eta>\xi\text{ and }\eta\in\Lambda\dnw X\,,\\
 \ob\,,&\text{if }\eta>\xi\text{ and }\eta\notin\Lambda\dnw X\,,\\
 \eta_X\,,&\text{if }\eta<\xi\text{ and }\eta\in\Lambda\upw X\,,\\
 \zb\,,&\text{if }\eta<\xi\text{ and }\eta\notin\Lambda\upw X
 \end{cases}
 \]
(we refer the reader to Section~\ref{Su:PosLat} for the notations~$\eta^X$, $\eta_X$).
Evidently, $f$ is isotone. In particular, $\cL_\xV(f)$ is an endomorphism of~$\cL_\xV(\Lambda)$.

{}From $f\res_X=\id_X$ and the assumption that~$X$ is a support of~$I$ it follows that $\cL_\xV(f)(I)=I$. On the other hand, as $\Lambda\setminus\set{\xi}$ is a support of~$I$ and $f(\Lambda\setminus\set{\xi})$ is contained in $(X\setminus\set{\xi})\cup\set{\zb,\ob}$, $X\setminus\set{\xi}$ is a support of $\cL_\xV(f)(I)$ (as in the proof of Lemma~\ref{L:ImSupp}). The conclusion follows.
\end{proof}

As every element of $\cL_\xV(\Lambda)$ has a finite support, we obtain immediately the following.

\begin{corollary}\label{C:SmSupp}
Let $\Lambda$ be a chain. Then every element $I\in\cL_\xV(\Lambda)$ has a smallest \pup{for containment} support, that we shall denote by $\supp I$ and call \emph{the support of~$I$}. Furthermore, $\supp I$ is finite.
\end{corollary}

We can now prove the main result of this section. The $\FF$-algebra with quasi-inversion~$R_\FF$ is defined in Section~\ref{S:CoordAl1} (cf. Figure~\ref{Fig:CommSq}).

\begin{theorem}\label{T:NoBanMeas}
Let $\FF$ be a countable field and let $\xV$ be a variety of $\FF$-algebras with quasi-inversion containing~$R_\FF$ as an element. Then there exists no $\cL_\xV(\omega_1)$-valued \Ban\ measure on the subset $\cX_\FF:=\setm{\txi\cdot\cR_\xV(\omega_1)}{\xi<\omega_1}$.
\end{theorem}

\begin{proof}
The structure $T:=\cR_\xV(\omega_1)$ is a regular $\FF$-algebra with quasi-inversion. Let~$\st$ be a term of~$\Sigma_\FF$ with arity~$n$, let~$\Lambda$ be a chain, and let~$X=\set{\xi_1,\dots,\xi_n}$ with all~$\xi_i\in\Lambda$ and $\xi_1<\cdots<\xi_n$. We shall write 
 \[
 \st[X]:=\st(\txi_1,\dots,\txi_n)\quad
 \text{evaluated in }\cR_\xV(\Lambda)\,.
 \]
Similarly, if $n=k+l$, $X=\set{\xi_1,\dots,\xi_k}$ with $\xi_1<\cdots<\xi_k$, and $Y=\set{\eta_1,\dots,\eta_l}$ with $\eta_1<\cdots<\eta_l$, we shall write
 \[
 \st[X;Y]:=\st(\txi_1,\dots,\txi_k,\teta_1,\dots,\teta_l)\quad
 \text{evaluated in }\cR_\xV(\Lambda)\,.
 \]
If $Y=\set{\eta_1,\dots,\eta_n}$ with $\eta_1<\cdots<\eta_n$ and $a\in\cR_\xV(\Lambda)$, we shall write
 \[
 \st[a;Y]:=\st(a,\teta_1,\dots,\teta_n)\quad
 \text{evaluated in }\cR_\xV(\Lambda)\,.
 \]
Now let $\ominus$ be an $\cL_\xV(\omega_1)$-valued \Ban\ measure on~$\cX$.

For all $\alpha\leq\beta<\omega_1$, there are a finite subset~$S_{\alpha,\beta}$ of~$\omega_1$ and a term~$\st_{\alpha,\beta}$ of~$\Sigma_\FF$ such that
 \begin{equation}\label{Eq:betalphdiff}
 \tbeta\cdot T\ominus\talpha\cdot T=
 \st_{\alpha,\beta}[S_{\alpha,\beta}]\cdot T\,. 
 \end{equation}
As $x\cdot T=(xx')\cdot T$ for each $x\in T$, we may assume that the term~$\st_{\alpha,\beta}$ is \emph{strongly idempotent}.
By Lemma~\ref{L:Deltaplus}, for each $\alpha<\omega_1$, there are an uncountable subset~$W_\alpha$ and a finite subset~$Z_\alpha$ of~$\omega_1$ such that, setting $S'_{\alpha,\beta}:=S_{\alpha,\beta}\setminus Z_\alpha$,
 \begin{equation}\label{Eq:CondDefWalph}
 S_{\alpha,\beta}\cap S_{\alpha,\gamma}=Z_\alpha\text{ and }
 Z_\alpha<S'_{\alpha,\beta}<S'_{\alpha,\gamma}\,,\quad
 \text{for all }\beta<\gamma\text{ in }W_\alpha\,. 
 \end{equation}
As the similarity type~$\Sigma_\FF$ is countable, we may refine further the uncountable subset~$W_\alpha$ in such a way that $\st_{\alpha,\beta}=\st_\alpha=\text{constant}$, for all $\beta\in W_\alpha$.

Now let $\alpha\leq\beta<\omega_1$. Pick $\gamma,\delta\in W_\alpha$ such that $\beta<\gamma<\delta$. We compute
 \begin{align*}
 \tbeta\cdot T\ominus\talpha\cdot T&=
 \tbeta\cdot T\cap(\tgamma\cdot T\ominus\talpha\cdot T)
 &&(\text{by the second part of Lemma~\ref{L:BanMeasMod}})\\
 &=\tbeta\cdot T\cap\st_\alpha[S_{\alpha,\gamma}]\cdot T\,,
 \end{align*}
so, by using Lemma~\ref{L:Pmeet},
 \begin{equation}\label{Eq:bsmalphwithsm}
 \tbeta\cdot T\ominus\talpha\cdot T
 =\sm(\tbeta,\st_\alpha[S_{\alpha,\gamma}])\cdot T\,.
 \end{equation}
In particular, the support of $\tbeta\cdot T\ominus\talpha\cdot T$ (cf. Corollary~\ref{C:SmSupp}) is contained in $S_{\alpha,\gamma}\cup\set{\beta}$. Similarly, this support is contained in $S_{\alpha,\delta}\cup\set{\beta}$, and so, by~\eqref{Eq:CondDefWalph},
 \begin{equation}\label{Eq:supp(betalphminus)}
 \supp(\tbeta\cdot T\ominus\talpha\cdot T)\subseteq
 Z_\alpha\cup\set{\beta}\,. 
 \end{equation}
Now set $k_\alpha:=\card Z_\alpha$, for each $\alpha<\omega_1$, and define a new term~$\su_\alpha$ by
 \begin{equation}\label{Eq:Defnsu1}
 \su_\alpha(\sx,\sy_1,\dots,\sy_{k_\alpha}):=
 \sm\bigl(\sx,\st_\alpha(\sy_1,\dots,\sy_{k_\alpha},1,\dots,1)\bigr)\,,
 \end{equation}
where the number of occurrences of the constant~$1$ in the right hand side of~\eqref{Eq:Defnsu1} is equal to $\operatorname{arity}(\st_\alpha)-k_\alpha$. As~$\sm$ is strongly idempotent, so is~$\su_\alpha$.

\setcounter{claim}{0}
\begin{claim}\label{Cl:smallbetsmalph}
The equality
$\tbeta\cdot T\ominus\talpha\cdot T=
\su_\alpha[\tbeta;Z_\alpha]\cdot T$ holds for all $\alpha\leq\beta<\omega_1$ such that $Z_\alpha\subseteq\beta+1$.
\end{claim}

\begin{cproof}
Pick $\gamma\in W_\alpha$ such that $\beta<S'_{\alpha,\gamma}$ (by~\eqref{Eq:CondDefWalph}, this is possible) and define the isotone map $f\colon\omega_1\to\omega_1\sqcup\set{\ob}$ by the rule
 \[
 f(\xi):=\begin{cases}
 \xi&(\text{if }\xi\leq\beta)\,\\
 \ob&(\text{if }\xi>\beta)\,
 \end{cases}\ ,\quad\text{for each }\xi<\omega_1\,.
 \] 
Every element of $Z_\alpha\cup\set{\beta}$ lies below~$\beta$, thus it is fixed by~$f$, while~$f$ sends each element of $S'_{\alpha,\gamma}$ to~$\ob$.
Hence, by applying the morphism~$\cL_\xV(f)$ to each side of~\eqref{Eq:bsmalphwithsm} and by using the definition~\eqref{Eq:Defnsu1}, we obtain
 \[
 \cL_\xV(f)(\tbeta\cdot T\ominus\talpha\cdot T)=
 \su_\alpha[\tbeta;Z_\alpha]\cdot T\,.
 \]
On the other hand, as every element of $Z_\alpha\cup\set{\beta}$ is fixed by~$f$, it follows from~\eqref{Eq:supp(betalphminus)} that $\tbeta\cdot T\ominus\talpha\cdot T$ is fixed under~$\cL_\xV(f)$. The conclusion follows. 
\end{cproof}

As~$\su_\alpha$ is a strongly idempotent term, the element $e_\alpha:=\su_\alpha[1;Z_\alpha]$ is idempotent in~$T$.

\begin{claim}\label{Cl:ealphaCompl}
The relation $T=\talpha\cdot T\oplus e_\alpha\cdot T$ holds for each $\alpha<\omega_1$.
\end{claim}

\begin{cproof}
Let $\beta<\omega_1$ with $\alpha<\beta$ and $Z_\alpha<\beta$, and define an isotone map $g\colon\omega_1\to\omega_1\sqcup\set{\ob}$ by the rule
 \[
 g(\xi):=\begin{cases}
 \xi&(\text{if }\xi<\beta)\\
 \ob&(\text{if }\xi\geq\beta)
 \end{cases}\ ,\quad
 \text{for each }\xi<\omega_1\,.
 \]
{}From Claim~\ref{Cl:smallbetsmalph} it follows that
$\tbeta\cdot T=\talpha\cdot T\oplus\su_\alpha[\tbeta;Z_\alpha]\cdot T$, thus, applying the $0$-lattice homomorphism~$\cL_\xV(g)$, we obtain
 \begin{equation}
 T=\talpha\cdot T\oplus\su_\alpha[1;Z_\alpha]\cdot T
 =\talpha\cdot T\oplus e_\alpha\cdot T\,.\tag*{\qedc}
 \end{equation}
\renewcommand{\qedc}{}
\end{cproof}

\begin{claim}\label{Cl:Monealpha}
The containment $e_\beta\cdot T\subseteq e_\alpha\cdot T$ holds, for all $\alpha\leq\beta<\omega_1$.
\end{claim}

\begin{cproof}
Pick $\gamma<\omega_1$ such that $\beta<\gamma$ and $Z_\alpha\cup Z_\beta<\gamma$. We compute
 \begin{align*}
 \su_\beta[\tgamma;Z_\beta]\cdot T&=\tgamma\cdot T\ominus\tbeta\cdot T
 &&(\text{by Claim~\ref{Cl:smallbetsmalph}})\\
 &\subseteq\tgamma\cdot T\ominus\talpha\cdot T
 &&(\text{by the monotonicity assumption on }\ominus)\\
 &=\su_\alpha[\tgamma;Z_\alpha]\cdot T
 &&(\text{by Claim~\ref{Cl:smallbetsmalph}})\,,
 \end{align*}
thus, as $\su_\alpha[\tgamma;Z_\alpha]$ is idempotent,
 \begin{equation}\label{Eq:subetaalpha}
 \su_\beta[\tgamma;Z_\beta]=
 \su_\alpha[\tgamma;Z_\alpha]\cdot\su_\beta[\tgamma;Z_\beta]\,.
 \end{equation}
Now define an isotone map $h\colon\omega_1\to\omega_1\sqcup\set{\ob}$ by the rule
 \[
 h(\xi):=\begin{cases}
 \xi&(\text{if }\xi<\gamma)\\
 \ob&(\text{if }\xi\geq\gamma)
 \end{cases}\ ,\quad
 \text{for each }\xi<\omega_1\,.
 \]
By applying $\cR_\xV(h)$ to the equation~\eqref{Eq:subetaalpha}, we obtain that $e_\beta=e_\alpha\cdot e_\beta$. The conclusion follows.
\end{cproof}

By Claims~\ref{Cl:ealphaCompl} and~\ref{Cl:Monealpha}, the family $\famm{e_\alpha\cdot T}{\alpha<\omega_1}$ defines a partial \Ban\ function on $\setm{\talpha\cdot T}{\alpha<\omega_1}$ in $\cL_\xV(\omega_1)=\LL(\cR_\xV(\omega_1))$. This contradicts the result of Proposition~\ref{P:NoBan}(ii).
\end{proof}

\section{A non-coordinatizable lattice with a large $4$-frame}\label{S:NonCoord}

A weaker variant of J\'onsson's Problem, of finding a non-coordinatizable \scml\ with a large $4$-frame, asks for a \emph{diagram counterexample} instead of an \emph{object counterexample}. In order to solve the full problem, we shall first settle the weaker version, by finding an $\omega_1$-indexed diagram of $4/5$-entire countable \scml s that cannot be lifted with respect to the~$\LL$ functor (cf. Lemma~\ref{L:NonLLLiftDiagr}).

The full solution of J\'onsson's Problem will then be achieved by invoking a tool from \emph{category theory}, introduced in Gillibert and Wehrung~\cite{GiWe}, designed to turn diagram counterexamples to object counterexamples. This tool is called there the ``Condensate Lifting Lemma'' (CLL). The general context of CLL is the following. We are given \emph{categories}~$\cA$, $\cB$, $\cS$ together with \emph{functors} $\Phi\colon\cA\to\cS$ and $\Psi\colon\cB\to\cS$, such that for ``many'' objects~$A\in\cA$, there exists an object~$B\in\cB$ such that $\Phi(A)\cong\Psi(B)$. We are trying to find an assignment $\Gamma\colon\cA\to\cB$, ``as functorial as possible'', such that $\Phi\cong\Psi\Gamma$ on a ``large'' subcategory of~$\cA$. Roughly speaking, CLL states that if the initial categorical data can be augmented by subcategories $\cA^\dagger\subseteq\cA$ and $\cB^\dagger\subseteq\cB$ (the ``small objects'') together with $\cS^\Rightarrow\subseteq\cS$ (the ``double arrows'') such that $(\cA,\cB,\cS,\Phi,\Psi,\cA^\dagger,\cB^\dagger,\cS^\Rightarrow)$ forms a \emph{projectable larder}, then this can be done. Checking larderhood, although somehow tedious, is a relatively easy matter, the least trivial point, already checked in~\cite{GiWe}, being the verification of the L\"owenheim-Skolem Property~$\mathrm{LS}^{\mathrm{r}}_{\aleph_1}(B)$ (cf. the proof of Lemma~\ref{L:RightLard}).

Besides an infinite combinatorial lemma by Gillibert, namely~\cite[Proposition~4.6]{Gill07}, we shall need only a small part of~\cite{GiWe}; basically, referring to the numbering used in version~1 of~\cite{GiWe} (which is the current version as to the present writing),
\begin{itemize}
\item[---] The definition of a projectability witness (Definition~1-5.1 in~\cite{GiWe}).

\item[---] The definition of a projectable larder (Definition~3-4.1 in~\cite{GiWe}). Strong larders will not be used.

\item[---] The statement of CLL (Lemma~3-4.2 in~\cite{GiWe}), for $\lambda=\mu=\aleph_1$. This statement involves the category $\mathbf{Bool}_P$ (Definition~2-2.3 in~\cite{GiWe}), here for $P:=\omega_1$, and the definition of $\boldsymbol{B}\otimes\vec A$ for $\boldsymbol{B}\in\mathbf{Bool}_P$ and a $P$-indexed diagram~$\vec A$. These constructions are rather easy and only a few of their properties, recorded in Chapter~2 of~\cite{GiWe}, will be used. A full understanding of \emph{lifters}, or of the $P$-scaled Boolean algebra~$\mathbf{F}(X)$ involved in the statement of CLL, is not needed.

\item[---] Parts of Chapter~6 in~\cite{GiWe}, that are, essentially, easy categorical statements about regular rings.

\end{itemize}

We shall consider the similarity type $\Gamma:=(0,\vee,\wedge,\sa_0,\sa_1,\sa_2,\sa_3,\sc_1,\sc_2,\sc_3,\sI)$, where $0$, $1$, the $\sa_i$s, and the~$\sc_i$s are symbols of constant, both~$\vee$ and~$\wedge$ are symbols of binary operations, and~$\sI$ is a (unary) predicate symbol. Furthermore, we consider the axiom system~$\cT$ in~$\Gamma$ that states the following:
\begin{itemize}
\item[(LAT)] $(0,\vee,\wedge)$ defines a \scml\ structure;

\item[(HOM)] $(\sa_0,\sa_1,\sa_2,\sa_3)$ is independent and~$\sa_0\sim_{\sc_i}\sa_i$ for each $i\in\set{1,2,3}$;

\item[(ID)] $\sI$ is an ideal;

\item[(REM)] every element of~$\sI$ is subperspective to~$\sa_0$ and disjoint from $\bigoplus_{i=0}^3\sa_i$;

\item[(BASE)] every element lies below $x\oplus\bigoplus_{i=0}^3\sa_i$ for some~$x\in\sI$.
\end{itemize}

In particular, (the underlying lattice of) every model for~$\cT$ is $4/5$-entire (cf. Definition~\ref{D:n/mentire}), so it has a large~$4$-frame.

Observe that every axiom of~$\cT$ has the form $(\forall\vec\sx)\bigl(\varphi(\vec\sx)\Rightarrow(\exists\vec\sy)\psi(\vec\sx,\vec\sy)\bigr)$ for finite conjunctions of atomic formulas~$\varphi$ and~$\psi$. For example, the axiom (REM) can be written
 \[
 (\forall\sx)\Bigl(\sI(\sx)\Rightarrow
 \bigl(\sx\wedge(\sa_0\vee\sa_1\vee\sa_2\vee\sa_3)=0
 \text{ and }(\exists\sy)(\sx\wedge\sy=\sa_0\wedge\sy=0
 \text{ and }\sx\leq\sa_0\vee\sy)\bigr)\Bigr)\,.
 \]
It follows that the category~$\cA$ of all models of~$\cT$, with their homomorphisms, is closed under arbitrary products and direct limits (i.e., directed colimits) of models.

Denote by~$\cS$ the category of all sectionally complemented modular lattices with $0$-lattice homomorphisms, and denote by~$\Phi$ the forgetful functor from~$\cA$ to~$\cS$.

Denote by~$\cB$ the category of all von Neumann regular rings with ring homomorphisms, and take~$\Psi:=\LL$, which is indeed a functor from~$\cB$ to~$\cS$.

Denote by~$\cA^\dagger$ (respectively, $\cB^\dagger$) the full subcategory of~$\cA$ (respectively, $\cB$) consisting of all \emph{countable} structures.

Denote by~$\cS^\Rightarrow$ the category of all sectionally complemented modular lattices with \emph{surjective} $0$-lattice homomorphisms. The morphisms in~$\cS^\Rightarrow$ will be called the \emph{double arrows} of~$\cS$.

Our first categorical statement about the data just introduced involves the \emph{left larders} developed in~\cite[Section~3.8]{GiWe}.

\begin{lemma}\label{L:LeftLard}
The quadruple $(\cA,\cS,\cS^\Rightarrow,\Phi)$ is a left larder.
\end{lemma}

\begin{proof}
We recall that left larders are defined by the following properties:
\begin{itemize}
\item[(CLOS$(\cA)$)] $\cA$ has all small directed colimits;

\item[(PROD$(\cA)$)] $\cA$ has all finite nonempty products;

\item[(CONT($\Phi$)] $\Phi$ preserves all small directed colimits;

\item[(PROJ$(\Phi,\cS^\Rightarrow)$)] $\Phi$ sends any extended projection of~$\cA$ (i.e., a direct limit $p=\varinjlim_{i\in I}p_i$ for projections $p_i\colon X_i\times Y_i\onto X_i$ in~$\cA$) to a double arrow in~$\cS$.
\end{itemize}

All the corresponding verifications are straightforward (e.g., every extended projection~$f$ is surjective, thus $\Phi(f)$ is a double arrow).
\end{proof}

Our second categorical statement states something about the more involved notion, defined in~\cite[Section~3.8]{GiWe}, of a \emph{right $\lambda$-larder}. We shall also use the notions, introduced in that paper, of \emph{projectability} of right larders. The following result is a particular case, for $\lambda=\aleph_1$, of Theorem~6-2.2 in (version~1 of)~\cite{GiWe}.

\begin{lemma}\label{L:RightLard}
Denote by $\cS^\dagger$ the class of all countable \scml s. Then the $6$-uple $(\cB,\cB^\dagger,\cS,\cS^\dagger,\cS^\Rightarrow,\LL)$ is a projectable right $\aleph_1$-larder.
\end{lemma}

\begin{proof}
Right larderhood amounts here to the conjunction of the two following statements:
\begin{itemize}
\item PRES${}_{\aleph_1}(\cB^\dagger,\LL)$: The lattice~$\LL(B)$ is ``weakly $\aleph_1$-presented'' in~$\cS$ (which means, here, \emph{countable}), for each $B\in\cB^\dagger$.

\item $\mathrm{LS}^{\mathrm{r}}_{\aleph_1}(B)$ (for every object~$B$ of~$\cB$): For every countable \scml\ $S$, every surjective lattice homomorphism\linebreak $\psi\colon\LL(B)\onto S$, and every sequence $\famm{u_n\colon U_n\mono B}{n<\omega}$ of monomorphisms in~$\cB$ with all~$U_n$ countable, there exists a monomorphism $u\colon U\mono\nobreak B$ in~$\cB$, lying above all~$u_n$ in the subobject ordering, such that~$U$ is countable and $\psi\circ\LL(u)$ is surjective.
\end{itemize}

Both statements are verified in~\cite[Chapter~6]{GiWe}.
\end{proof}

Now bringing together Lemmas~\ref{L:LeftLard} and~\ref{L:RightLard} is a trivial matter:

\begin{corollary}\label{C:Larder}
The $8$-uple $(\cA,\cB,\cS,\cA^\dagger,\cB^\dagger,\cS^\Rightarrow,\Phi,\LL)$ is a projectable $\aleph_1$-larder.
\end{corollary}

The following crucial result makes an essential use of our work on \Ban\ functions in Section~\ref{S:CoordAl1}.

\begin{lemma}\label{L:NonLLLiftDiagr}
There are increasing $\omega_1$-chains $\vec{A}=\famm{A_\xi}{\xi<\omega_1}$ and\linebreak $\vec{A}'=\famm{A'_\xi}{\xi<\omega_1}$ of countable models in~$\cA$, all with a unit, such that the following statements hold:
\begin{enumerate}
\item $\Phi\vec{A}$ cannot be lifted, with respect to the~$\LL$ functor, by any diagram in~$\cB$.

\item $A_\xi$ is a principal ideal of~$A'_\xi$, for each $\xi<\omega_1$.

\item All the models $A'_\xi$ share the same spanning $5$-frame.
\end{enumerate}
\end{lemma}

\begin{proof}
We fix a countable field~$\FF$ and we define regular $\FF$-algebras with quasi-inversion by $R_\xi:=\cR_\FF(\xi)$ (as defined in the comments just before Proposition~\ref{P:Index3}) and $S_\xi:=R_\xi^{5\times 5}$, for any ordinal~$\xi$. We set $R:=R_{\omega_1}$ and $S:=S_{\omega_1}$, and we identify~$R_\xi$ with its canonical image in~$R$, for each $\xi<\omega_1$ (this requires Proposition~\ref{P:cRonetoone}).
We denote by $\famm{e_{i,j}}{0\leq i,j\leq 4}$ the canonical system of matrix units of~$S$, so $\sum_{0\leq i\leq 4}e_{i,i}=1$ and $e_{i,j}e_{k,l}=\delta_{j,k}e_{i,l}$ (where~$\delta$ denotes the Kronecker symbol) in~$S$, for all $i,j,k,l\in\set{0,1,2,3,4}$.

We denote by $\psi:=(\famm{e_{i,i}S}{0\leq i\leq 4},\famm{(e_{i,i}-e_{0,i})S}{1\leq i\leq 4})$ the canonical spanning $5$-frame of~$\LL(S)$. Furthermore, we set $e:=\sum_{0\leq i\leq 3}e_{i,i}$, $b:=e_{4,4}$, and $b_\xi:=\txi\cdot b$ for each $\xi<\omega_1$. Observe that~$e$, $b$, and all~$b_\xi$ are idempotent, and that $1=e\oplus b$ and $b_\xi\utr b$ in~$S$. We set $U_\xi:=(e+b_\xi)S$, for each $\xi<\omega_1$, and
 \begin{align*}
 A'_\xi&:=\text{canonical copy of }\LL\bigl((R_{\xi+1})^{5\times 5}\bigr)
 \text{ in }\LL\bigl(R^{5\times 5}\bigr)\,,\\
 A_\xi&:=\text{ideal of }A'_\xi\text{ generated by }U_\xi\,,
 \end{align*}
for each $\xi<\omega_1$. In particular, $A'_\xi$ is a countable complemented sublattice of~$\LL(S)$ containing~$\psi$ while $A_\xi$ contains $\phi:=(\famm{e_{i,i}S}{0\leq i\leq 3},\famm{(e_{i,i}-e_{0,i})S}{1\leq i\leq 3})$, the canonical spanning $4$-frame of the principal ideal~$\LL(S)\dnw eS$.

In each~$A_\xi$, we interpret the constant~$\sa_i$ by~$e_{i,i}S$, for $0\leq i\leq 3$, and the constant~$\sc_i$ by~$(e_{i,i}-e_{0,i})S$, for $1\leq i\leq 3$. Furthermore, we interpret the predicate symbol~$\sI$ of~$\Gamma$ in each~$A'_\xi$ by $A'_\xi\dnw bS$, and in each~$A_\xi$ by~$A_\xi\dnw b_\xi S$. It is straightforward to verify that we thus obtain increasing $\omega_1$-chains~$\vec{A}$ and~$\vec{A}'$ of countable models in~$\cA$.

We claim that there is no $\LL(S)$-valued \Ban\ measure on $\setm{U_\xi}{\xi<\omega_1}$. Suppose otherwise. As $U_\xi=eS\oplus b_\xi S$ and $b_\xi S\subseteq bS$, with $eS\oplus bS=S$ in~$\LL(S)$, there exists, by Lemma~\ref{L:Transfer}, an $\bigl(\LL(S)\dnw bS\bigr)$-valued \Ban\ measure on $\setm{b_\xi S}{\xi<\omega_1}$. However, it follows from \cite[Lemma~10.2]{Jons62} that $\LL(S)\dnw bS$ is isomorphic to~$\LL(R)$, \emph{via} an isomorphism that sends~$b_\xi S$ to~$\txi R$, for each $\xi<\omega_1$. Thus there exists an $\LL(R)$-valued \Ban\ measure on $\setm{\txi R}{\xi<\omega_1}$. This contradicts Theorem~\ref{T:NoBanMeas}.

Any lifting of~$\vec{A}$, with respect to the functor~$\LL$, in~$\cB$ arises from an $\omega_1$-chain
 \[
 B_0\subset B_1\subset\cdots\subset B_\xi\subset\cdots
 \]
of regular rings, and it can be represented by the commutative diagram of Figure~\ref{Fig:LLLift}, for a system $\famm{\eps_\xi}{\xi<\omega_1}$ of isomorphisms.
\begin{figure}[htb]
 \[
 \def\labelstyle{\displaystyle}
 \xymatrix{
 A_0\ar@{^(->}[r] & A_1\ar@{^(->}[r] &\cdots&\cdots \ar@{^(->}[r]
 &A_\xi\ar@{^(->}[r] & \cdots\\
 \LL(B_0)\ar[u]^{\eps_0}_{\cong}\ar@{^(->}[r] &
 \LL(B_1)\ar[u]^{\eps_1}_{\cong}\ar@{^(->}[r] &\cdots&\cdots\ar@{^(->}[r]
 &\LL(B_\xi)\ar[u]^{\eps_\xi}_{\cong}\ar@{^(->}[r] & \cdots
 }
 \]
\caption{A lifting of~$\Phi\vec{A}$ with respect to~$\LL$}
\label{Fig:LLLift}
\end{figure}
It follows from Lemma~\ref{L:L(R)unital} that~$B_\xi$ is unital, for each $\xi<\omega_1$. Denote by $1_\xi$ the unit of~$B_\xi$, and set
 \[
 U_\beta\ominus U_\alpha:=\eps_\beta\bigl((1_\beta-1_\alpha)
 \cdot B_\beta\bigr)\,,\quad\text{for all }\alpha\leq\beta<\omega_1\,.
 \]
Let $\alpha\leq\beta\leq\gamma<\omega_1$.
{}From the commutativity of the diagram in Figure~\ref{Fig:LLLift} it follows that $U_\alpha=\eps_\beta(1_\alpha\cdot B_\beta)$. Hence, by applying the lattice isomorphism~$\eps_\beta$ to the relation $B_\beta=1_\alpha\cdot B_\beta\oplus(1_\beta-1_\alpha)\cdot B_\beta$, we obtain the relation $U_\beta=U_\alpha\oplus(U_\beta\ominus U_\alpha)$. Furthermore, from $1_\alpha\utr 1_\beta\utr 1_\gamma$ it follows that $1_\gamma-1_\alpha=(1_\gamma-1_\beta)\oplus(1_\beta-1_\alpha)$ in $\Idemp B_\gamma$, thus $(1_\gamma-1_\alpha)\cdot B_\gamma=(1_\gamma-1_\beta)\cdot B_\gamma\oplus(1_\beta-1_\alpha)\cdot B_\gamma$ in~$\LL(B_\gamma)$, thus, applying~$\eps_\gamma$ to each side of that relation, we obtain
 \begin{align*}
 U_\gamma\ominus U_\alpha&=(U_\gamma\ominus U_\beta)\oplus
 \eps_\gamma\bigl((1_\beta-1_\alpha)\cdot B_\gamma\bigr)\\
 &=(U_\gamma\ominus U_\beta)\oplus
 \eps_\beta\bigl((1_\beta-1_\alpha)\cdot B_\beta\bigr)
 &&(\text{see Figure~\ref{Fig:LLLift}})\\
 &=(U_\gamma\ominus U_\beta)\oplus(U_\beta\ominus U_\alpha)\,.
 \end{align*}
Therefore, $\ominus$ defines an $\LL(S)$-valued \Ban\ measure on $\setm{U_\xi}{\xi<\omega_1}$, which we just proved impossible.
\end{proof}

Observe that all the $A'_\xi$s share the same unit, while the $\omega_1$-sequence formed with all the units of the~$A_\xi$s is increasing.

\begin{theorem}\label{T:4/5NonCoord}
There exists a non-coordinatizable, $4/5$-entire \scml\ $L$ of cardinality~$\aleph_1$, which is in addition isomorphic to an ideal in a \cml~$L'$ with a spanning $5$-frame \pup{so~$L'$ is coordinatizable}.
\end{theorem}

\begin{proof}
We use the notation and terminology of Gillibert and Wehrung~\cite{GiWe}. It follows from Gillibert~\cite[Proposition~4.6]{Gill07} that there exists an $\aleph_1$-lifter $(X,\bdX)$ of the chain~$\omega_1$ such that $\card X=\aleph_1$.

Consider the diagrams~$\vec{A}$ and~$\vec{A}'$ of Lemma~\ref{L:NonLLLiftDiagr}, and observe that both~$A_\xi$ and~$A'_\xi$ belong to~$\cA^\dagger$ (i.e., they are countable), for each~$\xi<\omega_1$. We form the \emph{condensates}
 \[
 L:=\Phi\bigl(\xF(X)\otimes\vec{A}\bigr)\quad\text{and}\quad
 L':=\Phi\bigl(\xF(X)\otimes\vec{A}'\bigr)\,.
 \]
{}From $\card X\leq\aleph_1$ it follows that the $\omega_1$-scaled Boolean algebra~$\xF(X)$ is the directed colimit of a direct system of at most~$\aleph_1$ finitely presented objects in the category $\mathbf{Bool}_{\omega_1}$. It follows that $\card L\leq\aleph_1$ and $\card L'\leq\aleph_1$. We shall prove that~$L$ is not coordinatizable; in particular, by \cite[Theorem~10.3]{Jons62}, $\card L=\aleph_1$.

Suppose that there exists an isomorphism~$\chi\colon\LL(B)\to L$, for some regular ring~$B$. By CLL (cf. \cite[Lemma~3-4.2]{GiWe}) together with Corollary~\ref{C:Larder}, there exists an $\omega_1$-indexed diagram~$\vec{B}$ in~$\cB$ such that $\LL\vec{B}\cong\Phi\vec{A}$. This contradicts Lemma~\ref{L:NonLLLiftDiagr}. Therefore, $L$ is not coordinatizable.

Furthermore, $\xF(X)\otimes\vec{A}$ is a direct limit of finite direct products of the form $\prod_{i=1}^nA_{\xi_i}$, where the shape of the indexing system depends only on~$X$. As~$A_\xi$ is an ideal of~$A'_\xi$ for each $\xi<\omega_1$, $\prod_{i=1}^nA_{\xi_i}$ is an ideal of $\prod_{i=1}^nA'_{\xi_i}$ at each of those places. Therefore, taking direct limits, we obtain that $\xF(X)\otimes\vec{A}$ is isomorphic to an ideal of $\xF(X)\otimes\vec{A}'$, so~$L$ is an ideal of~$L'$. As the class of all lattices with a spanning $5$-frame is closed under finite products and directed colimits and as all~$A'_\xi$s have a spanning $5$-frame, $L'$ also has a spanning $5$-frame.
\end{proof}

Theorem~\ref{T:4/5NonCoord} provides us with a non-coordinatizable ideal in a coordinatizable \cml\ of cardinality~$\aleph_1$. We do not know whether an ideal in a \emph{countable} coordinatizable \scml\ is coordinatizable.

As the lattice~$L$ of Theorem~\ref{T:4/5NonCoord} is $4/5$-entire and sectionally complemented, it has a large $4$-frame. Hence it solves negatively the problem, left open in J\'onsson~\cite{Jons62}, whether a \scml\ with a large $4$-frame is coordinatizable.

\begin{remark}\label{Rk:PpalIdl4fr}
As the lattice~$L$ of Theorem~\ref{T:4/5NonCoord} has a large $4$-frame, every principal ideal of~$L$ is coordinatizable. Indeed, fix a large $4$-frame $\alpha=(a_0,a_1,a_2,a_3,c_1,c_2,c_3)$ in~$L$ and put $a:=\bigoplus_{i=0}^3a_i$. Every principal ideal~$I$ of~$L$ is contained in~$L\dnw b$ for some~$b\in L$ such that $a\leq b$. As~$\alpha$ is a large $4$-frame of the \cml~$L\dnw b$ and by \cite[Theorem~8.2]{Jons60}, $L\dnw b$ is coordinatizable. As~$I$ is a principal ideal of~$L\dnw b$, it is also coordinatizable (cf. \cite[Lemma~10.2]{Jons62}). 
\end{remark}

\begin{remark}\label{Rk:ClosIssues}
It is proved in Wehrung~\cite{CXCoord} that the union of a chain of coordinatizable lattices may not be coordinatizable. The lattices considered there are $2$-distributive with unit. Theorem~\ref{T:4/5NonCoord} extends this negative result to lattices (without unit) with a large $4$-frame. Furthermore, it also shows that an ideal in a coordinatizable lattice~$L'$  may not be coordinatizable, even in case~$L'$ has a spanning $5$-frame. By contrast, it follows from \cite[Lemma~10.2]{Jons62} that any \emph{principal} ideal of a coordinatizable lattice is coordinatizable. It is also observed in \cite[Proposition~3.5]{CXCoord} that the class of coordinatizable lattices is closed under homomorphic images, reduced products, and taking neutral ideals.
\end{remark}

It is proved in Wehrung~\cite{CXCoord} that the class of all coordinatizable lattices \emph{with unit} is not first-order. The lattices considered there are $2$-distributive (thus without non-trivial homogeneous sequences) with unit. The following result extends this negative result to the class of all lattices (without unit) admitting a large $4$-frame.

\begin{corollary}\label{C:4/5NonCoord1}
The class of all coordinatizable \scml s with a large $4$-frame is not first-order definable.
\end{corollary}

\begin{proof}
Fix a large $4$-frame $\alpha=\bigl((a_0,a_1,a_2,a_3),(c_1,c_2,c_3)\bigr)$ in the lattice~$L$ of Theorem~\ref{T:4/5NonCoord}, and put $a:=a_0\oplus a_1\oplus a_2\oplus a_3$. As~$L$ is $4/5$-entire, it satisfies the first-order statement, with parameters from $\set{a_0,a}$,
 \begin{equation}\label{Eq:weak4/5}
 (\forall\sx)(\exists\sy)(\sx\leq a\oplus\sy\text{ and }
 \sy\lesssim a_0)\,.
 \end{equation}
Let~$K$ be a countable elementary sublattice of~$L$ containing all the seven entries of~$\alpha$. As~$L$ satisfies~\eqref{Eq:weak4/5}, so does~$K$, thus~$\alpha$ is a large~$4$-frame in~$K$. It follows from \cite[Theorem~10.3]{Jons62} that~$K$ is coordinatizable. On the other hand, $L$ is not coordinatizable and~$K$ is an elementary sublattice of~$L$.
\end{proof}

The following definition is introduced in~\cite[Definition~5.1]{BanFct1}.

\begin{definition}\label{D:BanTail}
A \emph{\Ban\ trace} on a lattice~$L$ with zero is a family\linebreak $\famm{a_i^j}{i\leq j\text{ in }\Lambda}$ of elements in~$L$, where~$\Lambda$ is an upward directed partially ordered set with zero, such that
\begin{enumerate}
\item $a_i^k=a_i^j\oplus a_j^k$ for all $i\leq j\leq k$ in~$\Lambda$;

\item $\setm{a_0^i}{i\in\Lambda}$ is cofinal in~$L$.
\end{enumerate}
\end{definition}

We proved in~\cite[Theorem~6.6]{BanFct1} that \emph{A \scml\ with a large $4$-frame is coordinatizable if{f} it has a \Ban\ trace}. Hence we obtain the following result.

\begin{corollary}\label{C:NoBanTrace}
There exists a $4/5$-entire \scml\ of cardinality~$\aleph_1$ without a \Ban\ trace.
\end{corollary}

\section{Acknowledgment}\label{S:Acknow}
I thank Luca Giudici for his many thoughtful and inspiring comments on the paper, in particular for his example quoted in Remark~\ref{Rk:IncromCh}.


\begin{thebibliography}{99}
\bibitem{Bana}
B. Banaschewski,
Totalgeordnete Moduln (German),
Arch. Math.~\textbf{7} (1957), 430--440.

\bibitem{Birk94}
G. Birkhoff,
Lattice Theory, Corrected reprint of the 1967 third edition. American Mathematical Society Colloquium Publications, \textbf{25}. American Mathematical Society, Providence, R.I., 1979.

\bibitem{Bkou70}
R. Bkouche,
Puret\'e, mollesse et paracompacit\'e,
C. R. Acad. Sci. Paris S\'er. A-B~\textbf{270} (1970), A1653--A1655.

\bibitem{BuSa}
S. Burris and H.\,P. Sankappanavar,
A Course in Universal Algebra, The Millennium Edition, online
manuscript available at \texttt{http://www.thoralf.uwaterloo.ca}\,,
1999. xvi+315~p. (Previously published as: Graduate Texts in
Mathematics, \textbf{78}. New York, Heidelberg, Berlin:
Springer-Verlag.)

\bibitem{FiRo70}
R.\,L. Finney and J. Rotman,
Paracompactness of locally compact Hausdorff spaces,
Michigan Math. J.~\textbf{17}, no.~4 (1970), 359--361.

\bibitem{FrHa54}
K.\,D. Fryer and I. Halperin,
Coordinates in geometry,
Trans. Roy. Soc. Canada. Sect. III. (3)~\textbf{48} (1954), 11--26.

\bibitem{FrHa56}
K.\,D. Fryer and I. Halperin,
The von~Neumann coordinatization theorem for complemented
modular lattices, Acta Sci. Math. (Szeged) \textbf{17} (1956),
203--249.

\bibitem{Gill07}
P. Gillibert,
Critical points of pairs of varieties of algebras, Internat. J. Algebra Comput.~\textbf{19}, no.~1 (2009), 1--40.

\bibitem{GiWe}
P. Gillibert and F. Wehrung,
{}From objects to diagrams for ranges of functors, preprint 2010, arXiv:1003.4850.

\bibitem{Good91}
K.\,R. Goodearl,
Von Neumann Regular Rings,
Second edition.  Robert E. Krieger Publishing Co., Inc., Malabar, FL, 1991.

\bibitem{GoMM}
K.\,R. Goodearl, P. Menal, and J. Moncasi,
Free and residually Artinian regular rings, J. Algebra~\textbf{156} (1993), 407--432.

\bibitem{GLT2}
G. Gr\"atzer,
General Lattice Theory, second edition. Birkh\"auser Verlag, Basel,
1998.

\bibitem{Halp61}
I. Halperin,
A simplified proof of von Neumann's coordinatization theorem,
Proc. Nat. Acad. Sci. U.S.A.~\textbf{47} (1961), 1495--1498.

\bibitem{Halp81}
I. Halperin,
von Neumann's coordinatization theorem,
C. R. Math. Rep. Acad. Sci. Canada~\textbf{3}, no.~5 (1981), 285--290.

\bibitem{Halp83}
I. Halperin,
von Neumann's coordinatization theorem, 
Acta Sci. Math. (Szeged)~\textbf{45} (1983), no.~1-4, 213--218.

\bibitem{Herr}
C. Herrmann,
Generators for complemented modular lattices and  
the von Neumann-J\'{o}nsson Coordinatization Theorems, Algebra Universalis~\textbf{63}, no.~1 (2010), 45--64.

\bibitem{HeSe}
C. Herrmann and M. Semenova,
Existence varieties of regular rings and complemented modular lattices, J. Algebra~\textbf{314}, no.~1 (2007), 235--251.

\bibitem{Jech}
T. Jech,
Set Theory, Pure and Applied Mathematics, Academic Press [Harcourt Brace Jovanovich, Publishers], New York - London, 1978.

\bibitem{Jons60}
B. J\'onsson,
Representations of complemented modular lattices,
Trans. Amer. Math. Soc.~\textbf{60} (1960), 64--94.

\bibitem{Jons62}
B. J\'onsson,
Representations of relatively complemented modular lattices,
Trans. Amer. Math. Soc.~\textbf{103} (1962), 272--303.

\bibitem{Kapl58}
I. Kaplansky,
On the dimension of modules and algebras, X. A right hereditary ring which is not left hereditary, Nagoya Math. J.~\textbf{13} (1958), 85--88.

\bibitem{Maed58}
F. Maeda,
Kontinuierliche Geometrien (German),
Die Grundlehren der mathematischen Wissenschaften in Einzeldarstellungen mit besonderer Ber\"ucksichtigung der Anwendungsgebiete, Bd.~\textbf{95}. Springer-Verlag, Berlin - G\"ottingen - Heidelberg, 1958. x+244~p.

\bibitem{Malc}
A.\,I. Mal'cev,
Algebraic Systems (Algebraicheskie sistemy) (Russian)
Sovremennaja Algebra. Moskau: Verlag ``Nauka'', Hauptredaktion
f\"ur physikalisch-mathematische Literatur, 1970.
English translation: Die Grundlehren der mathematischen Wissenschaften. Band \textbf{192}. Berlin - Heidelberg - New York: Springer-Verlag; Berlin: Akademie-Verlag, 1973.

\bibitem{Micol}
F.~Micol,
On representability of $\ast$-regular rings and modular ortholattices,
PhD thesis, TU Darmstadt, January 2003. Available online at\newline \texttt{http://elib.tu-darmstadt.de/diss/000303/diss.pdf}\,.

\bibitem{Neum}
J. von Neumann,
Continuous geometry, Princeton Mathematical Series, No.~\textbf{25}. Princeton University Press, Princeton, N.J. 1960.

\bibitem{SaSo}
M. Saarim\"aki and P. Sorjonen,
On Banaschewski functions in lattices, Algebra Universalis \textbf{28}, no.~1 (1991), 103--118.

\bibitem{CXCoord}
F. Wehrung,
Von Neumann coordinatization is not first-order, J. Math. Log.~\textbf{6}, no.~1 (2006), 1--24.

\bibitem{BanFct1}
F. Wehrung,
Coordinatization of lattices by regular rings without unit and Banaschewski functions, Algebra Universalis, to appear.

\end{thebibliography}
\end{document}